\documentclass[graybox]{SNmult}

\usepackage{type1cm}
\usepackage{makeidx}
\usepackage{graphicx}
\usepackage{multicol}
\usepackage[bottom]{footmisc}
\usepackage{newtxtext}
\usepackage[varvw]{newtxmath}
\usepackage{amsmath,mathtools}
\usepackage{bm}
\usepackage{booktabs}
\usepackage{array}
\usepackage{xcolor}
\usepackage{float}
\usepackage{etoolbox}

\newcommand{\R}{\mathbb{R}}
\newcommand{\Ewr}{\mathcal E_{wr}}
\newcommand{\Mfast}{\mathcal M_0}

\usepackage{comment}
\usepackage{hyperref}
\hypersetup{
	unicode,
	colorlinks=true,
	linktoc=all,
	linkcolor=blue,
	breaklinks,
	urlcolor=cyan   
}
\makeindex

\AtBeginEnvironment{align}{\small}
\AtBeginEnvironment{align*}{\small}
\AtBeginEnvironment{equation}{\small}
\AtBeginEnvironment{multline}{\small}
\allowdisplaybreaks

\begin{document}
	
	\title*{Water-at-Rest Equilibrium Stability Analysis of a first-moment Shallow Water Exner Moment Model with Sediment Entrainment and Deposition: Extended Technical Report}
	\titlerunning{Water-at-Rest Equilibrium Stability of SWEMED1}
	\author{Afroja Parvin\orcidID{0009-0008-3270-6442} and\\ Giovanni Samaey\orcidID{0000-0001-8433-4523} and\\ Julian Koellermeier\orcidID{0000-0002-8822-461X}}
	\authorrunning{A. Parvin et al.}
	\institute{Afroja Parvin \at Department of Computer Science, KU Leuven, Belgium; School of Mathematical Sciences, Peking University, China, \email{afroja.parvin@kuleuven.be}
		\and Giovanni Samaey \at Department of Computer Science, KU Leuven, Belgium
		\and Julian Koellermeier \at Department of Mathematics, Computer Science and Statistics, Ghent University, Belgium;\\ Bernoulli Institute, University of Groningen, Netherlands}
	
	\maketitle

\abstract{We derive the first-moment Shallow Water Exner Moment model with sediment entrainment and deposition (SWEMED1) and show that the full source term has a fully-settled water-at-rest equilibrium manifold. We prove that the model is only weakly hyperbolic at this equilibrium, which prevents the use of Yong's structural stability framework. However, a linear spectral analysis and numerical results do not indicate instability. Based on numerical results, we introduce a fast-slow scaling of the source term and for the fast limit we derive a new suspended water-at-rest equilibrium manifold, which has a different structure, but is still only weakly hyperbolic. Our results show that the remaining obstruction is linked to the transport closure of the SWEMED1 and we give a constructive direction for the derivation of new closures leading to models with more desirable analytical properties.
}

\section{Introduction}
Sediment transport in shallow-water flows is important for describing the evolution of riverbeds, estuaries, and coastal regions \cite{gonzalez2020robust}. Standard depth-averaged models couple the shallow water equations with the Exner equation and, when needed, a suspended-concentration equation \cite{audusse2004fast,del2023lagrange,exner1925uber,gonzalez2020robust,meng2020localized,zhao2019depth}. These models are computationally efficient, but they use a single depth-averaged velocity and therefore do not resolve the vertical structure of the horizontal flow. This requires quantities related to the near-bed flow, such as bottom shear stress, bedload transport, and entrainment, to be represented by empirical closures based on averaged variables.

The Shallow Water Moment (SWM) model derived in \cite{CiCP-25-669} improves this description by expanding the horizontal velocity in vertical direction using a Legendre basis and deriving evolution equations for the corresponding moment coefficients. Hyperbolic versions of these moment models were developed in \cite{koellermeier2020analysis}, and moment models for bedload morphodynamics were studied in \cite{Garres}, leading to the Shallow Water Exner Moment model (SWEM). In this work, we extend a special case of the SWEM (i.e. $N=1$ moments) to derive and analyze the first-moment Shallow Water Exner Moment model with sediment entrainment and deposition, denoted by SWEMED1. The new SWEMED1 model couples water mass balance, depth-averaged momentum, the first velocity-moment equation, suspended sediment concentration, and bed evolution. The suspended concentration feeds back into the hydrodynamics through a depth-averaged mixture density.

The first aim of this extended report is to give a self-contained derivation of SWEMED1. For non-vanishing bottom friction, moment relaxation, and settling velocity, at equilibrium, we then show that the SWEMED1 source term forces the mean velocity, the first velocity-moment coefficient, and the suspended concentration to vanish, giving rise to the fully-settled water-at-rest equilibrium manifold.
We then examine the stability of solutions around this equilibrium using Yong's structural stability framework \cite{Yong1999SingularPerturbations}, following the use of this framework for hyperbolic shallow water moment equations in \cite{huang2022equilibrium}. Related stability questions for hyperbolic moment systems have also been studied in the context of globally
hyperbolic moment closures and kinetic moment models \cite{Zhao2017StabilityGHM,Cai2013GloballyHyperbolicGrad,Di2017LinearStabilityHMM}. While the source Jacobian has the desired dissipative structure, we show that the transport part is only weakly hyperbolic at the fully-settled equilibrium manifold. This prevents the verification of the full set of Yong's structural stability conditions at this manifold. Since Yong's conditions are sufficient rather than necessary, this obstruction should not be interpreted as instability of the model. In fact, a direct linear spectral calculation shows no growing normal modes, and a numerical relaxation test is consistent with this non-growing behavior, but indicates different time scales in the relaxation towards equilibrium.
For the full source term relaxing to the fully-settled water-at-rest equilibrium manifold, hydrodynamic relaxation and entrainment--deposition balance together lead to a state with no suspended sediment. Numerically, however, the solution can pass through an intermediate stage in which the hydrodynamic variables are already close to rest while the suspended concentration is still positive and decays more slowly. This is sometimes also referred to as a secular equilibrium \cite{Prince1979SecularEquilibrium}.
We therefore introduce a fast-slow source splitting to derive a suspended water-at-rest equilibrium manifold that describes intermediate states in which hydrodynamic variables have relaxed while suspended sediment remains.
We show that the suspended water-at-rest equilibrium has a different source structure, but by itself does not remove the weak hyperbolicity obstruction when the suspended concentration is transported with the vanishing depth-averaged velocity. This shows that the remaining obstruction is linked to the transport closure in the suspended-concentration equation, rather than to the source splitting. A clear direction for future research would therefore be a suspended-concentration transport closure, for instance, through an effective sediment transport velocity. However, a complete analysis of such modified closures is left outside the scope of the present work.

The rest of the paper is organized as follows. Section~\ref{sec:hswemed1_derivation} derives the SWEMED1 model. Section~\ref{sec:water_rest_stability} studies the fully-settled water-at-rest equilibrium manifold and its stability using Yong's structural stability conditions. Section~\ref{sec:numerical_relaxation} presents a numerical relaxation test for the evolution of SWEMED1 near the fully-settled water-at-rest manifold. Section~\ref{sec:fast_manifold} introduces the fast-slow scaling, derives the limiting fast manifold, and shows that the obstruction persists. Section~\ref{sec:conclusion} concludes the paper and gives directions for future work.

    \section{A first-moment Shallow water moment model for sediment transport}\label{sec:hswemed1_derivation}
    In this section, we provide a self-contained derivation of the first-moment Shallow Water Exner Moment model $(N=1)$ with sediment Entrainment and Deposition, denoted by SWEMED1. The model extends the shallow water moment framework for bedload morphodynamics \cite{Garres} by adding a depth-averaged suspended-sediment concentration equation together with entrainment-deposition exchange between the bed and the water column. The hydrodynamic variables are the water height $h$, the depth-averaged horizontal velocity $u_m$, and the first velocity-moment coefficient $\alpha_1$, while the morphodynamic variables are the suspended concentration $c_m$ and the bed elevation $h_b$. These components are coupled through the mixture-density dependence, the bedload flux in the Exner equation, and the entrainment--deposition source terms. 
    
    \subsection{Hydrodynamic equations}\label{Hydrodynamic}
    We consider the two-dimensional inhomogeneous Navier--Stokes equations
    to derive the coupled mathematical model for sediment
    transport. However, the derivation can be extended to three spatial
    dimensions \cite{CiCP-25-669}. The inhomogeneous Navier--Stokes equations play a central
    role in the description of geophysical flows, including rivers and
    shallow coastal currents. These flows are incompressible, but the
    density varies due to the presence of suspended sediment. Two-dimensional
    inhomogeneous Navier--Stokes equations are written as
    \begin{align}
    	\partial_x u + \partial_z w & = 0, \label{eq:2D NSE system_1}\\
    	\partial_t (\rho u)+ \partial_x (\rho u^2) + \partial_z (\rho uw)
    	&= -\partial_x p + \partial_x \sigma_{xx} + \partial_z \sigma_{xz},
    	\label{eq:2D NSE system_2} \\
    	\partial_t (\rho w)+ \partial_x (\rho uw) + \partial_z (\rho w^2)
    	&= -\partial_z p + \partial_x \sigma_{zx} + \partial_z \sigma_{zz}
    	- \rho g,
    	\label{eq:2D NSE system_3} \\
    	\partial_t \rho + \partial_x (\rho u) + \partial_z (\rho w)
    	&= 0,
    	\label{eq:2D NSE system_4}
    \end{align}
    where $u$ and $w$ denote the horizontal and vertical velocity
    components, respectively, $g$ is the gravitational acceleration, and
    $\sigma_{xx},\sigma_{xz},\sigma_{zx},\sigma_{zz}$ are the components of
    the deviatoric stress tensor.
    
    Due to the water-sediment mixture, the density $\rho$ is not constant and is
    written as
    \begin{equation}\label{density equation}
    	\rho(t,x,z)= \rho_w + c(t,x,z)\,(\rho_s-\rho_w),
    \end{equation}
    where $c(t,x,z)\in[0,1]$ is the local volumetric sediment concentration,
    and $\rho_w$ and $\rho_s$ are the densities of water and sediment,
    respectively. The mixture as a whole is assumed incompressible and
    satisfies $\nabla\cdot \boldsymbol{u} = 0$. The variations of the density $\rho$ are solely due to the variations of the concentration $c$.
    
    In principle, both $\rho$ and $c$ depend on the vertical coordinate
    $z$, and this dependence should be taken into account when deriving a
    reduced model. In the present work, we restrict ourselves to both depth-averaged density $\rho$ and depth-averaged concentration $c_m$. This is motivated by our attention to dilute and moderately dilute suspensions $c \ll 1$, where the vertical variations of the density $\rho$ are
    small compared with the overall vertical variation of the velocity $u$. Accordingly, in the shallow water framework we approximate the mixture density $\rho$ by its vertical average and treat it as a function of the horizontal coordinate only, i.e. $\rho(t,x)$. More precisely, instead of \eqref{density equation}, we use 
    \begin{equation}\label{eq:averaged_density}
    	\rho(t,x) = \rho_w + c_m(t,x)\,(\rho_s-\rho_w),
    \end{equation}
    and with a slight abuse of notation, use the same symbol $\rho$ for
    this vertically averaged density in the reduced model. The reduced hydrodynamics therefore use the depth-averaged density $\rho(c_m)$. Note that extensions towards vertically changing densities and concentrations are left for future work and are closely linked to existing non-hydrostatic
    models \cite{Scholz2024DispersionSME}. When near-bed sediment exchange is evaluated, see Section \ref{bedflux}, the near-bed concentration is closed diagnostically through $c_b=S_b c_m$ \eqref{bed_concentration}; this closure is not used to introduce a vertically varying density field in the hydrodynamic pressure law.
    \par Under the assumption of shallowness (vertical accelerations are small compared with gravity) we adopt the hydrostatic approximation for the pressure $p$ and, consistent with the weakly stratified regime described
    above, we approximate the mixture density $\rho$ by its vertical average.
    In particular, density variations enter the reduced hydrodynamics through $\rho=\rho(c_m)$ in the hydrostatic pressure and the associated baroclinic forcing, while
    higher-order effects of vertical density stratification on the inertial
    terms are neglected. With these modelling assumptions, the
    two-dimensional inhomogeneous Navier--Stokes system
    \eqref{eq:2D NSE system_1}--\eqref{eq:2D NSE system_4} reduces to
    \begin{align}
    	\partial_x u + \partial_z w & = 0, \label{reference system21}\\
    	\partial_t u + \partial_x (u^2) + \partial_z (u w)
    	&= - \dfrac{1}{\rho}\,\partial_x p
    	+ \dfrac{1}{\rho}\,\partial_z \sigma_{xz}, \label{reference system22}\\
    	\partial_t c + \partial_x (c u) + \partial_z (c w)
    	&= 0.  \label{reference system23}
    \end{align}
    Consistent with the hydrostatic approximation, the pressure is
    expressed as
    \begin{equation}\label{pressure equation}
    	p(t,x,z)
    	= \rho(t,x)\,g\big(h(t,x)+ h_b(t,x) - z\big),
    \end{equation}
    where $h(t,x):=h_s(t,x)-h_b(t,x)$ is the water depth, $h_b$ denotes the
    bed elevation and $h_s$ the free-surface elevation. The flow domain is
    bounded below by the bottom topography $h_b(t,x)$ and above by the free
    surface $h_s(t,x)$. In this formulation the dominant feedback of the suspended sediment on the hydrodynamics occurs through the mixture density $\rho=\rho(c_m)$ in the hydrostatic pressure \eqref{pressure equation} and the buoyancy terms in the later depth-averaged momentum equation, while a possible dependence of the viscous and turbulent stresses on the concentration $c$ is neglected. Crucially, the suspended concentration is not merely a passive scalar in the reduced model: it affects the momentum equation through the mixture-density dependence of the hydrostatic pressure.
    \par The boundary conditions for the system include kinematic boundaries applied both at the free surface and the bottom, which are defined as 
    \begin{align}
    	\partial_t h_s + \left.u \right\vert_{z=h_s}\partial_x h_s-\left. w \right\vert_{z=h_s}	& = 0, \label{bc_z_cord_1}\\
    	\partial_t h_b + \left.u \right\vert_{z=h_b}\partial_x h_b-\left. w \right\vert_{z=h_b}& = -\, F_b, \label{bc_z_cord_2}
    \end{align}
    where the surface flux is zero, and the bed exchange rate $F_b$, which couples the model to the morphology, will be explained in Section~\ref{bedflux}. Following classical entrainment-deposition modelling, we represent the exchange by a net interfacial flux, as commonly used in depth-averaged formulations. Existing shallow water moment models simplify the problem by neglecting the bed exchange rate, i.e., by taking $F_b=0$ \cite{Garres,CiCP-25-669}.

    \par Additionally, we assume that the remaining shear stress component of the deviatoric stress tensor, \( \sigma_{xz} \), vanishes at the free surface and follows Manning friction \( \tau \), at the bottom \cite{Garres}. As a result, the boundary conditions for the shear stress are
    \begin{equation}\label{shear_stress_z_cord}
    	\left.\sigma_{xz}\right\vert_{z=h_s}=0 \quad \quad \text{and} \quad \quad \left.\sigma_{xz}\right\vert_{z=h_b}=\, \tau.
    \end{equation}
    
    \subsubsection{Mapped reference system and moment expansion}
    While the system in equations \eqref{reference system21}–\eqref{reference system23} already defines our reference system, this system still depends on the vertical dimension and we aim to reduce its dimension and derive a reduced model by adopting the so-called moment approach outlined in \cite{CiCP-25-669}. This moment approach is based on two main ideas: first, mapping to a reference (\(\zeta-\)coordinate) system and second, expanding the horizontal velocity using a moment expansion. To derive our coupled model, we adopt the approach of \cite{CiCP-25-669} but adapt it to our specific setting, which incorporates coupling with the morphology discussed later in Section~\ref{morphology}. 
    
    According to the first main idea of the Shallow Water Moment model (SWM) approximation described in \cite{CiCP-25-669} we introduce a scaled vertical variable $\zeta(t,x)$, which transforms the $z$-coordinate from the physical space $z\in \left[h_b,h_s\right]$ to the mapped space $\zeta\in \left[0,1\right]$, as shown in Figure~\ref{fig:mapping_graph}. The definition of the mapping to $\zeta$-coordinates is given by 
    \begin{equation}\label{mapping}
    	\zeta(t,x) = \dfrac{z-h_b(t,x)}{h_s(t,x)-h_b(t,x)}= \dfrac{z-h_b(t,x)}{h(t,x)}.
    \end{equation}

    \begin{figure}[H] 
    	\centering
    	\includegraphics[width=0.95\textwidth]{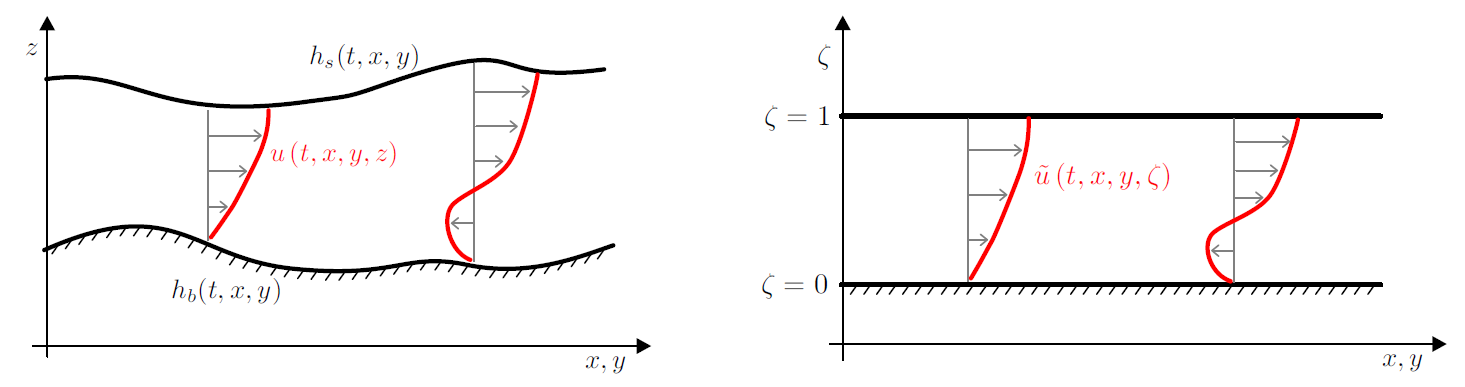}
    	\caption{The mapping from physical $z$-space to transformed $\zeta$-space \cite{CiCP-25-669}.}
    	\label{fig:mapping_graph}
    \end{figure}
    Using \eqref{mapping}, for any function $\Phi(t,x,z)$, the corresponding mapped function in $\zeta$-coordinates is given by
    \begin{equation}
    	\Tilde{\Phi}(t,x,\zeta) = \Phi(t,x,\zeta h(t,x)+h_b(t,x)).
    \end{equation}
    The corresponding differential operators read
    \begin{equation}\label{diffop1}
    	\partial_\zeta \Tilde{\Phi} = h\partial_z\Phi \quad \text{and} \quad h\partial_s\Phi = \partial_s(h\Tilde{\Phi})-\partial_\zeta(\partial_s(\zeta h+h_b)\Tilde{\Phi}), \quad \text{for}\,\,\, s \in \left[t,x\right]. 
    \end{equation}
    Taking into account the mapping \eqref{mapping} and using the differential operators \eqref{diffop1}, the complete vertically resolved system can be written as \cite{CiCP-25-669}
    \begin{align}
    	\partial_x \left(h \Tilde{u}\right)+\partial_\zeta \left[\Tilde{w}-\Tilde{u}\partial_x(\zeta h+h_b) \right]   &= 0, \label{eq:resolved  system1_1}\\ 
    	\begin{split}
    		\partial_t (h \Tilde{u})+ \partial_x(h \Tilde{u}^2) +  \partial_\zeta \left[\Tilde{u}\left(\Tilde{w} -\Tilde{u}\partial_x\left(\zeta h+h_b\right)- \partial_t\left(\zeta h+h_b\right)\right)\right] \\
    		+gh\partial_x (h + h_b) + \dfrac{gh^2}{\rho}(1-\zeta)\partial_x \rho & = \dfrac{1}{\rho}\partial_\zeta \Tilde{\sigma}_{xz},
    	\end{split} \label{eq:resolved  system1_2} \\
    	\partial_t (h \Tilde{c})  + \partial_x \left(h \Tilde{c}\Tilde{u}\right) +\partial_\zeta \left[\Tilde{c} \left(\Tilde{w}- \partial_x(\zeta h+h_b)\Tilde{u}-\partial_t(\zeta h+h_b) \right) \right]&=0. \label{eq:resolved  system1_3}
    \end{align}
    The system of equations \eqref{eq:resolved  system1_1} to \eqref{eq:resolved  system1_3} is referred to as the vertically resolved system because it incorporates the dependence on the vertical variable $\zeta.$\\
    Furthermore, the boundary conditions \eqref{bc_z_cord_1} and \eqref{bc_z_cord_2} and also the shear stress at the free surface and bottom are mapped to $\zeta$-coordinate as
    \begin{align}
    	& \partial_t h_s + \left.\Tilde{u} \right\vert_{\zeta=1} \partial_x h_s-\left. \Tilde{w} \right\vert_{\zeta=1} = 0, \quad \quad \quad \quad \quad \quad  \quad \left.\Tilde{\sigma}_{xz}\right\vert_{\zeta=1}=0, \label{bc_zeta_1}\\
    	&  \partial_t h_b + \left.\Tilde{u}\right\vert_{\zeta=0}  \partial_x h_b-\left. \Tilde{w} \right\vert_{\zeta=0}= - F_b, \quad \quad \quad \quad \quad \quad \left.\Tilde{\sigma}_{xz}\right\vert_{\zeta=0}=\, \tau. \label{bc_zeta_2}
    \end{align}
    \par Following the second main idea of the Shallow Water Moment model in \cite{CiCP-25-669} it is possible to expand the horizontal velocity in the vertical variable $\zeta$, using a general Legendre polynomial expansion up to general polynomial degree $N$, in the same way as in \cite{CiCP-25-669}. In this work, we retain only one velocity-moment, i.e. $N=1$. The extension to higher-order moment systems follows the same projection strategy as in \cite{CiCP-25-669}, but is not needed for the equilibrium analysis below and would only complicate its presentation. For $N=1$ the velocity profile is written as
    \begin{equation}\label{moment expansion}
    	\Tilde{u}(t,x,\zeta) = u_m(t,x) + \alpha_1(t,x) \phi_1(\zeta),
    \end{equation}
    where $u_m(t,x)= \displaystyle \int_{0}^{1}\Tilde{u}(t,x,\zeta) d\,\zeta $ is the mean of the horizontal velocity, and the basis function $\phi_1(\zeta):[0,1]\rightarrow \mathbb{R} $ is the scaled linear Legendre polynomials of degree $N=1$ given by
    \begin{equation*}
    	\phi_1(\zeta)= 1-2\zeta.
    \end{equation*}
    We note that $\phi_1$ satisfies the condition $\phi_1(0)=1$, and $\displaystyle \int_{0}^{1} \phi_1(\zeta)d\zeta = 0$, meaning that it is orthogonal to the constant function.
    The corresponding coefficient $\alpha_1(t,x)$ is unknown and an additional equation is derived by projecting the momentum balance equation onto the first Legendre polynomial, see \cite{CiCP-25-669}. The coefficient $\alpha_1(t,x)$ is the first velocity-moment coefficient, also called first moment. Since $\phi_1(\zeta)=1-2\zeta$ is linear in the vertical coordinate, the term $\alpha_1\phi_1(\zeta)$ represents the linear component of the reconstructed vertical velocity profile. Thus $\alpha_1$ measures the deviation from a vertically uniform, depth-averaged velocity profile.
    
    \par By employing the first-moment approach with $N=1$ outlined above within our specific framework, we need to derive the evolution equations for the extended set of conservative variables \( (h, hu_m, h\alpha_1, h c_m)^T \). The evolution equations for the water height \( h \), the mean horizontal velocity \( u_m \), and the volumetric sediment concentration \( c_m \) are obtained through the following Galerkin projection of \eqref{eq:resolved system1_1}, \eqref{eq:resolved system1_2}, and \eqref{eq:resolved system1_3}, respectively
    \begin{equation}\label{Galerkin projection}
    	\langle \cdot, \, 1 \rangle = \int_{0}^{1} \cdot \, \, d\zeta.
    \end{equation} \\
    Furthermore, the additional evolution equation for the first velocity-moment coefficient $\alpha_1$ is derived by employing a first-order Galerkin projection of the momentum balance equation \eqref{eq:resolved system1_2}
    \begin{equation}\label{Higher-order Galerkin projection}
    	\langle \cdot, \, \phi_1 \rangle = \int_{0}^{1} \cdot \, \phi_1(\zeta) \, d\zeta.
    \end{equation} 
    i.e. multiplying equation \eqref{eq:resolved system1_2} by the corresponding linear test function \( \phi_1(\zeta) \) and subsequently integrating with respect to \( \zeta. \)  
    \par In the three subsequent sections, we provide a detailed derivation of the depth-averaged equations for \( h \), \( hu_m \), and \( h\alpha_1 \). The corresponding derivation for \( h c_m \) is presented within the morphodynamic Section~\ref{concentration_avg} since morphodynamic change employs suspended sediment concentration.  
    
    \subsubsection{Depth-averaging the mass balance}
    To recover an explicit expression for $ \Tilde{w}$, equation \eqref{eq:resolved  system1_1} can be written in the following integral form \cite{CiCP-25-669} 
    \begin{equation}\label{vertical velocity}
    	\Tilde{w}= - \partial_x \left(h \int_{0}^{\zeta} \Tilde{u} \,d\hat{\zeta}\right)+ \Tilde{u}\partial_x(\zeta h+h_b),
    \end{equation}
    Now, to derive the standard depth-averaged mass balance equation of the shallow water system, we depth-average the transformed mass balance equation, i.e. we apply the projection \eqref{Galerkin projection} to the equation \eqref{vertical velocity} and use the kinematic boundary conditions \eqref{bc_zeta_1} - \eqref{bc_zeta_2}, which yields the following depth-averaged mass balance equation similar to \cite{CiCP-25-669}
    \begin{equation}\label{depth-average mass balance}
    	\partial_t h + \partial_x (hu_m) = F_b.
    \end{equation}
    Equation \eqref{depth-average mass balance} represents an equation for the water height $h$, and couples to the equation for the discharge $hu_m$, (Section~\ref{momentum_discharge}) and the bed exchange rate $F_b$ (Section~\ref{bedflux}). Without the bed exchange $F_b$, equation \eqref{depth-average mass balance} simplifies to the same as in the SWM \cite{CiCP-25-669}. 
    
    %%%%%%%%%%%%%%%%%%%% momentum balance%%%%%%%%%%%%%%%%%%%%
    \subsubsection{Depth-averaging the momentum equation}\label{momentum_discharge}
    The depth-averaged momentum balance equation along the horizontal $x-$axis is derived by applying \eqref{Galerkin projection} to \eqref{eq:resolved system1_2} and is given by 	
    \begin{align}\label{depth_avg_momentum}   
    	\begin{split}
    		\partial_t(h u_m) +\partial_x \left( h \left(u_m^2+ \dfrac{\alpha_1^2}{3}\right)+ \dfrac{gh^2}{2} \right)
    		&=  - gh \partial_x h_b + F_b u_b \\
    		&- \dfrac{gh^2}{2 \rho}(\rho_s-\rho_w)\partial_x c_m- \epsilon |u_b|u_b.
    	\end{split}
    \end{align}
    We provide the detailed derivation in Appendix \ref{averaged momentum balance}.
    \par Here we only discuss the last three terms on the right-hand side of \eqref{depth_avg_momentum}, as they are new compared to the momentum balance equation in \cite{CiCP-25-669}:
    \begin{enumerate}
    	\item The term $F_b u_b$ quantifies momentum transfer between flow and erodible bed due to sediment exchange.
    	\item The term $\displaystyle \dfrac{gh^2}{2 \rho}(\rho_s-\rho_w)\partial_x c_m$ represents the effects of spatial variations in sediment concentration.
    	\item The term $\epsilon |u_b|u_b$ represents bottom friction in the momentum balance. 
    	We consider Manning friction at the bottom \cite{Garres} and Newtonian friction within the fluid  \cite{Garres,CiCP-25-669}. Since the Galerkin projection of the momentum balance is influenced only by bottom friction and remains independent of friction within the fluid, we present only the bottom friction in this section, while the treatment of friction within the fluid will be addressed in the subsequent section.\\
    	Accordingly, the Manning friction at the bottom is given by
    	\begin{equation*}
    		\left. \Tilde{\sigma}_{xz}(\zeta) \right\vert_{\zeta=0}= -\rho \epsilon |u_b|u_b,
    	\end{equation*}
    	where the bottom velocity is expressed as
    	\begin{equation}\label{bottom_velocity}
    		u_b=\left. \Tilde{u}\right\vert_{\zeta=0}=u_m+\displaystyle \alpha_1(t,x)\, \phi_1(0)=u_m+\alpha_1(t,x),
    	\end{equation} 
    	and $\epsilon$ is a dimensionless constant defined in \cite{Garres}.
    \end{enumerate}

    %%%%%%%%%% Higher order averages %%%%%%%%%%%%%%%%%%%%%%
    \subsubsection{First-order moment of the momentum equation}\label{Higher order averages_1}
    The evolution equation for the first velocity-moment coefficient $\alpha_1$ is derived from the first-order Galerkin projection \eqref{Higher-order Galerkin projection} of \eqref{eq:resolved system1_2}. 
    Thus, the resulting equation for $h\alpha_1$ is given by
    \begin{align}\label{final_higher_average_equation}
    \begin{split}
        \partial_t (h\alpha_1)&+ \partial_x \left(2hu_m\alpha_1\right)
        = u_m\partial_x (h \alpha_1) + 2F_b\alpha_1 \\
        &\quad\quad -\dfrac{gh^2}{2\rho}(\rho_s-\rho_w)\partial_x c_m
        -3\epsilon |u_b|u_b - \dfrac{12\nu}{h}\alpha_1.
    \end{split}
    \end{align}
    where the already inserted constants generated by the scaled linear Legendre polynomial $\phi_1(\zeta)=1-2\zeta$ in the notation similar to  \cite{CiCP-25-669} are given by 
    \begin{equation}\label{constant matrices}
    \begin{aligned}
        A_{111}&=3\int_0^1\phi_1^3\,d\zeta=0, 
        &B_{111}&=3\int_0^1\phi_1'(\zeta)\left(\int_0^\zeta\phi_1(\hat\zeta)\,d\hat\zeta\right)\phi_1(\zeta)\,d\zeta=0,\\
        C_{11}&=\int_0^1(\phi_1')^2\,d\zeta=4,
        &G_{11}&=3\int_0^1\phi_1\phi_1'\,d\zeta=0,\\
        H_{11}&=3\int_0^1\zeta\phi_1\phi_1'\,d\zeta=1,
        &K_1&=\int_0^1\zeta\phi_1\,d\zeta=-\dfrac16.
    \end{aligned}
    \end{equation}
    We provide the detailed derivation of \eqref{final_higher_average_equation} in Appendix \ref{Higher order averages}.
    \par On the right-hand side of equation \eqref{final_higher_average_equation}, four additional terms appear compared to the classical SWM \cite{CiCP-25-669}: 
    \begin{enumerate}
    	\item The term $2F_b\alpha_1$ represents a source term associated with entrainment and deposition effects at the bottom, governed by the bed exchange rate $F_b.$
    	\item The term $-\dfrac{gh^2}{2\rho}(\rho_s-\rho_w)\partial_x c_m$ arises from the first-order Galerkin projection on the spatial variation of sediment concentration. This term signifies the interaction between suspended sediment and momentum balance.
    	\item The term $3\epsilon |u_b|u_b+12\nu\alpha_1/h$ results from the first-order Galerkin projection of Manning friction at the bottom, i.e., $\left. \Tilde{\sigma}_{xz}(\zeta) \right\vert_{\zeta=0}= -\rho \epsilon |u_b|u_b,$ and Newtonian friction within the fluid i.e.,  $\left. \Tilde{\sigma}_{xz}(\zeta) \right\vert_{\zeta \in (0,1)}= - \dfrac{\mu}{h} \partial_\zeta \Tilde{u}(\zeta)$, and the projection reads  
	    \begin{align*}
	    \dfrac{1}{\rho}\int_0^1 \phi_1\partial_\zeta\Tilde\sigma_{xz}\,d\zeta
	    &=\dfrac{1}{\rho}\int_0^1 \partial_\zeta(\phi_1\Tilde\sigma_{xz})\,d\zeta
	    -\dfrac{1}{\rho}\int_0^1 \Tilde\sigma_{xz}\partial_\zeta\phi_1\,d\zeta\\
	    &=- \dfrac{1}{\rho}\left.\phi_1\Tilde\sigma_{xz}\right\vert_{\zeta=0}
	    -\dfrac{\mu}{\rho h}\int_0^1 \phi_1^\prime\partial_\zeta\Tilde u\,d\zeta\\
	    &=-\epsilon |u_b|u_b
	    -\dfrac{\nu}{h}\int_0^1 \alpha_1(\phi_1^\prime)^2\,d\zeta\\
	    &=-\epsilon |u_b|u_b-\dfrac{4\nu}{h}\alpha_1 .
	    \end{align*}
    \end{enumerate}
 
    \subsection{Morphodynamic equations}\label{morphology}
    %%%%%%%%%% Average of volumetric sediment concentration %%%%%%%%%%%%%%%%%%%%%%
    The morphodynamic part of the model covers the suspended and bedload transport. It consists of the depth-averaged volumetric sediment concentration equation for $c_m$ and the Exner equation for $h_b$ as explained below.
    
    \subsubsection{Depth-averaged volumetric sediment concentration}
    \label{concentration_avg}
    In this section, we derive the depth-averaged equation for the suspended sediment concentration. The transported variable is the depth-averaged concentration $c_m(t,x)$. The near-bed concentration entering the deposition closure is prescribed algebraically through the Bradford factor, $c_b=S_b c_m$, without introducing an additional concentration profile, which is an extension left for future work.
    \ \
    In the mapped $\zeta$–coordinates, the vertically resolved concentration
    equation \eqref{eq:resolved system1_3} can be written in conservative form as
    \begin{equation}
    	\partial_t (h \tilde c)
    	+ \partial_x (h \tilde c \tilde u)
    	+ \partial_\zeta J = 0,
    	\label{eq:c_resolved_zeta}
    \end{equation}
    where the vertical sediment flux $J$ is given by
    \begin{equation}
    	J(t,x,\zeta)
    	= \tilde c(t,x,\zeta)\,
    	\Big[
    	\tilde w
    	- \partial_x(\zeta h + h_b)\,\tilde u
    	- \partial_t(\zeta h + h_b)
    	\Big].
    	\label{eq:J_def}
    \end{equation}
    Integrating \eqref{eq:c_resolved_zeta} over $\zeta \in [0,1]$ and using
    $\int_0^1 \partial_\zeta J \, d\zeta = J(1) - J(0)$ yields
    \begin{equation}
    	\partial_t \Big( h \int_0^1 \tilde c \, d\zeta \Big)
    	+ \partial_x \Big( h \int_0^1 \tilde c \tilde u \, d\zeta \Big)
    	+ \big[ J(1) - J(0) \big] = 0.
    	\label{eq:c_int}
    \end{equation}
    We define the depth-averaged volumetric sediment concentration and the depth-averaged concentration flux as
    \begin{equation}
    	c_m(t,x) := \int_0^1 \tilde c(t,x,\zeta)\,d\zeta,
    	\qquad
    	\langle \tilde c \tilde u \rangle
    	:= \int_0^1 \tilde c \tilde u\,d\zeta.
    \end{equation}
    With this notation, \eqref{eq:c_int} becomes
    \begin{equation}
    	\partial_t (h c_m)
    	+ \partial_x \big( h \langle \tilde c \tilde u \rangle \big)
    	= J(0) - J(1).
    	\label{eq:hcm_J}
    \end{equation}
    At the free surface the kinematic boundary condition
    $\partial_t h_s + \tilde u|_{\zeta=1}\,\partial_x h_s -
    \tilde w|_{\zeta=1} = 0$ implies that no sediment crosses the interface, so that $J(1) = 0.$\\
    At the bed, the kinematic condition for the moving interface reads
    \begin{equation}
    	\partial_t h_b + \tilde u|_{\zeta=0}\,\partial_x h_b
    	- \tilde w|_{\zeta=0} = -F_b,
    	\label{eq:bed_kinematic}
    \end{equation}
    with $F_b$ the bed exchange rate introduced in
    Section~\ref{bedflux}. If the bed were impermeable to sediment, the
    purely kinematic contribution associated with the moving interface would
    give $J(0) = \tilde c(t,x,0)\,F_b$. In the present setting, however, the
    bed acts as an active sediment reservoir: sediment particles can
    be entrained from the bed into the flow (entrainment) and deposited from
    suspension onto the bed (deposition). We therefore interpret the entrainment
    and deposition laws as prescribing the interfacial sediment flux
    on the fluid side and set
    \begin{equation}
    	J(0) = E - D,
    	\label{eq:J0_ED}
    \end{equation}
    where $E$ and $D$ are the entrainment and deposition rates defined in
    Section~\ref{bedflux}. Substituting \eqref{eq:J0_ED} and $J(1)=0$ into
    \eqref{eq:hcm_J} yields the depth-averaged concentration equation once the standard shallow-water suspended-load closure $\langle\tilde c\tilde u\rangle\simeq c_m u_m$ is used.
    
    \begin{equation}
    	\partial_t(h c_m) + \partial_x (h c_m u_m) = E - D.
    	\label{eq:hcm_final_LHS}
    \end{equation}

    The closure $\langle \tilde c\tilde u\rangle\simeq c_m u_m$ is the depth-averaged suspended-load approximation used in SWEMED1. It keeps the concentration equation compatible with a single transported scalar $c_m$ and uses the depth-averaged velocity $u_m$ as transport velocity. This choice is important in Section~\ref{sec:fast_manifold}: at water-at-rest, where $u_m=0$, the same closure determines the concentration row of the transport matrix and is directly related to the persistence of the zero-speed degeneracy at the fast manifold.

    %%%%%%%%%%%%%%%%%%%%%%%%%%%%%%%%%%%%%%%%%%%
    \subsubsection{Bedload mass balance equation}
    We use the Exner equation  \cite{exner1925uber} to model the evolution of bedload transport, incorporating the principle of mass conservation for the sediment layer. This formulation represents sediment transport through a flux term, and the right-hand side of the equation includes a source term that describes entrainment and deposition processes, which reads as 
    \begin{equation}\label{Exner equation}
    	\partial_t h_b  + \dfrac{1}{1-\psi}\partial_x Q_b = -F_b,
    \end{equation}
    where $h_b$ is the bed elevation, $Q_b$ is the solid sediment discharge, and $F_b$ is the bed exchange rate, which accounts for entrainment and deposition processes, and defined in Section~\ref{bedflux}. More explicitly, the Exner equation represents the rate of bed deformation. Active sediment transport and rapid bed deformation will occur if the flow entrains more sediment particles than it deposits.
    \par The result of the Exner equation \eqref{Exner equation} is highly dependent on the choice of the sediment discharge formula, $Q_b$. We note that there is a plethora of formulas for the sediment discharge term. In this paper, we adopt the following formula \cite{Garres}
    \begin{equation}\label{sediment discharge}
    	Q_b = sgn(\tau)\, Q \, \Phi(\theta), 
    \end{equation}
    where $Q=\sqrt{\left(\dfrac{\rho_s}{\rho_w}-1\right)g\,d_s^3}$ is the characteristic discharge and $sgn(\cdot)$ is the sign function.	
    We chose the expression \eqref{sediment discharge} for $Q_b$ since it depends on the bottom shear stress, used in its non-dimensional form $\theta$ and $Q_b$ can be expressed as a function of $\theta$ as $\Phi(\theta).$  The dimensionless bottom shear stress $\theta$ is also called Shields parameter and defined by
    \begin{equation}\label{Shields parameter}
    	\theta = \dfrac{|\tau|}{g\,(\rho_s- \rho_w)\,d_s}= \dfrac{\rho \epsilon |u_b| u_b}{g\,(\rho_s- \rho_w)\,d_s},
    \end{equation}
    where $\epsilon$ is a dimensionless constant \cite{Garres}, $u_b$ is the velocity at bottom, as defined in  \eqref{bottom_velocity}, and $d_s$ is the diameter of a sediment particle.
    \par There is a framework that includes many different formulas for $\Phi(\theta)$ (see, for example, \cite{gonzalez2020robust}). In this work, we employ the Meyer--Peter--Muller formula \cite{Garres}
    \begin{equation}\label{Meyer--Peter--Muller}
    	\Phi(\theta)= 8(\theta-\theta_{c})_{+}^{3/2},
    \end{equation} 
    where $(\cdot)_+$ is the positive part and $\theta_c$ is the critical shear stress. Notably, sediment transport is initiated only when the Shields parameter $\theta$ exceeds the critical threshold $\theta_c$.
    
   \begin{remark}
    We use the Meyer--Peter--Muller formula \eqref{Meyer--Peter--Muller} for the bedload flux \(Q_b\) \eqref{sediment discharge}, but evaluate the bed shear stress \(\rho \epsilon|u_b|u_b\) through the bottom velocity \(u_b\) \eqref{bottom_velocity} obtained from the moment reconstruction, rather than through the depth-averaged velocity \(u_m\).
    This follows the modelling strategy used for shallow water moment models with bedload transport in \cite{Garres}. The underlying use of velocity moments and friction closures is based on the shallow water moment framework of \cite{CiCP-25-669}. In classical shallow water--Exner models, the bedload flux \(Q_b\) \eqref{sediment discharge} is commonly written as a function of the Shields parameter or bed shear stress, which is then closed in terms of depth-averaged flow quantities \(u_m\) through a friction law \cite{Cordier2011BedloadSplitting,FernandezNieto2017FormalSVE,Gonzalez2020Robust}. Using the bottom velocity \(u_b\) instead emphasizes the role of the reconstructed vertical velocity
    profile in the moment model. A detailed recalibration of the empirical bedload and
    friction coefficients for this choice is left outside the scope of the present work.
    \end{remark}    
    \subsubsection{Morphological conditions}\label{bedflux}
    The bed exchange rate $F_b$, used in Section~\ref{Hydrodynamic} and Section~\ref{morphology}, represents a dynamic exchange between two effects:
    (1) the sediment deposition due to gravitational force and (2) the sediment entrainment from the interface due to erosion from the bottom sediment layer. This balance is crucial for understanding sedimentary processes and is defined by
    \begin{equation}\label{bedflux_eq}
    	F_b = \dfrac{E-D}{1-\psi},
    \end{equation}  
    Here $E$ and $D$ denote the entrainment and deposition rates, respectively, and $\psi$ is the bedload porosity.    
    \vspace{2mm}\\
    To close the system, we refer to \cite{bradford1999hydrodynamics,del2023lagrange,garcia1993experiments,gonzalez2020robust,zhang1993sedimentation} and take the following formulas for entrainment and deposition.\\
    The sediment entrainment and deposition follow from \cite{del2023lagrange,gonzalez2020robust}, as given by
    \begin{equation}\label{eq:ED_standard}
    	E = \omega_0 (1-\psi) E_s,\qquad D= \omega_0 c_b,
    \end{equation} 
    where
    \begin{itemize}
    	\item[-] $\omega_0$ is the settling velocity determined by \cite{zhang1993sedimentation}, 
    	\begin{equation}\label{settling velocity}
    		\omega_0 = \sqrt{\left(\dfrac{13.95\nu_w}{d_s}\right)^2+1.09\, 	\rho_w \left(\dfrac{\rho_s}{\rho_w}-1\right) gd_s} - \dfrac{13.95\nu_w}{d_s},
    	\end{equation}
    	\item[-] $\nu_w$ is the kinematic viscosity of water and $d_s$ is the diameter of sediment particles,
    	\item[-] $E_s$ is the sediment entrainment coefficient and computed by \cite{garcia1993experiments},
    	\begin{equation}
    		E_s = \dfrac{1.3\times 10^{-7} \mathcal{Z}^5}{1+4.3\times 10^{-7} 	\mathcal{Z}^5},
    	\end{equation} 
    	\item[-]  $\mathcal{Z} = \displaystyle  \dfrac{\gamma_1\sqrt{c_D} |u_b| }{\omega_0}\mathcal{R}^{\gamma_2}_{p}$ with particle Reynolds number	$\mathcal{R}_p = \displaystyle \dfrac{\sqrt{(\rho_s-\rho_w) g d_s}d_s}{\nu_w}$, $u_b$ the bottom velocity, and $c_D$ the bed drag coefficient,
    	\item[-] $\gamma_1,\gamma_2$ are two parameters depending on $\mathcal{R}_p$ \cite{gonzalez2020robust}
    	$$
    	\left(\gamma_1,\gamma_2\right)=\begin{cases}
    		\left(1,0.6\right), & \text{if $\mathcal{R}_p> 2.36$}\\
    		\left(0.586,1.23\right), & \text{if $\mathcal{R}_p \le 2.36$}
    	\end{cases}
    	$$
    	\item[-] $c_b$ is the fractional concentration of sediment suspension near the bed and defined by \cite{bradford1999hydrodynamics}
    	\begin{equation}\label{bed_concentration}
    		c_b = 	c_m(t,x)\,S_b; \qquad S_b=\left(0.4\left(\dfrac{d_s}{D_{sg}}\right)^{1.64}+1.64\right),
    	\end{equation}
    	with $c_m(t,x)$ the depth-averaged volumetric sediment concentration and $D_{sg}$ the geometric mean size of the suspended sediment mixture. In this work, we assume all particles are of equal size, i.e., $D_{sg}=d_s$.
    \end{itemize}
    \subsection{Coupled hydro-morphodynamic model}
    By incorporating the derivations and definitions outlined in Section~\ref{Hydrodynamic} and Section~\ref{morphology}, we obtain a closed first-moment Shallow Water Exner Moment model with Entrainment and Deposition (SWEMED1), where the $1$ denotes the one additional moment $\alpha_1$. In contrast to the existing SWEM in \cite{Garres}, the SWEMED1 (i) introduces a sediment concentration equation, (ii) couples the variable sediment–water mixture density with the momentum equation and first-order moment $\alpha_1$, and (iii) includes additional source terms arising from entrainment and deposition. The resulting system consists of a total of $5$ coupled equations \eqref{complete system} for conservative variables set $(h,hu_m,h\alpha_1,h c_m,h_b)^T\in \mathbb{R}^{5}$.
    \begin{equation*}
    	\underbrace{1}_{\text{mass balance}} + \underbrace{1}_{\text{momentum equation}} +\underbrace{1}_{\text{moment equation}}+ 
    	\textcolor{black}{							\underbrace{1}_{\text{sediment concentration}}}+
    	\textcolor{black}{							\underbrace{1}_{\text{bedload mass balance}}}= \, 5
    \end{equation*} 
    Thus the coupled model is formulated as follows		
    \begin{equation}\label{complete system}
    	\left\{
    	\begin{aligned}
    		&\partial_t h  + \partial_x (h u_m) 
    		&&= \dfrac{E-D}{1-\psi},\\
    		&\partial_t(h u_m)  +\partial_x \left( h \left(u_m^2+ \dfrac{\alpha_1^2}{3}\right) + \dfrac{gh^2}{2} \right)
    		&&= - gh \partial_x h_b - \dfrac{gh^2}{2 	\rho}(\rho_s-\rho_w)\partial_x c_m \\
    		&&&\quad + \dfrac{(E-D) u_b}{1-\psi}- \epsilon |u_b|u_b,\\
    					&\partial_t (h\alpha_1)+ \partial_x \left(2hu_m\alpha_1\right)  
    		&&= u_m\partial_x (h \alpha_1)
    		-\dfrac{gh^2}{2\rho}(\rho_s-\rho_w)\partial_x c_m \\
    		&&&\quad +\dfrac{2(E-D)}{1-\psi}\alpha_1
    		-3\epsilon |u_b|u_b - \dfrac{12\nu}{h}\alpha_1,\\
    		&\partial_t(h c_m) + \partial_x (h c_m u_m) 
    		&&= E-D,\\
    		&\partial_t h_b  + \displaystyle \partial_x \left(\dfrac{Q_b}{1-\psi}\right) 
    		&&= \dfrac{D-E}{1-\psi}.
    	\end{aligned}
    	\right.
    \end{equation}

    For conciseness, we write the SWEMED1 model defined in \eqref{complete system} in compact matrix-vector notation as
	\begin{equation}\label{eq:standard_hswemed1}
		\partial_t W+A(W)\partial_x W=S(W),
	\end{equation}
    where
    \begin{equation}\label{eq:W_standard}
	   W=(h,hu_m,h\alpha_1,hc_m,h_b)^T .
     \end{equation}
	The transport matrix \(A(W)\) of SWEMED1 \eqref{eq:standard_hswemed1} written with respect to the conservative variables \(W\) in \eqref{eq:W_standard}, is
	\begin{equation}\label{eq:A_standard_N1}
		A(W)=
		\begin{pmatrix}
			0&1&0&0&0\\[0.5mm]
			gh-u_m^2-\dfrac{\alpha_1^2}{3}-\dfrac{g h c_m(\rho_s-\rho_w)}{2\rho}
			&2u_m&\dfrac{2\alpha_1}{3}&\dfrac{g h(\rho_s-\rho_w)}{2\rho}&gh\\[0.5mm]
			-2\alpha_1u_m-\dfrac{g h c_m(\rho_s-\rho_w)}{2\rho}
			&2\alpha_1&u_m&\dfrac{g h(\rho_s-\rho_w)}{2\rho}&0\\[0.5mm]
			-c_m u_m&c_m&0&u_m&0\\[0.5mm]
			\delta_h&\delta_q&\delta_q&\delta_c&0
		\end{pmatrix}.
	\end{equation}
	Here \(\delta_h\), \(\delta_q\), and \(\delta_c\) denote derivatives of \(Q_b/(1-\psi)\) with respect to the conservative variables in \eqref{eq:W_standard}. More precisely, \(\delta_q\) denotes the derivative with respect to \(hu_m\). Since \(Q_b\) depends on \(hu_m\) and \(h\alpha_1\) through the same bottom velocity \(u_b=u_m+\alpha_1\) \eqref{bottom_velocity}, the same coefficient \(\delta_q\) appears in both the \(hu_m\) and \(h\alpha_1\)-columns. The coefficient \(\delta_c\) denotes the derivative with respect to \(hc_m\). For the Meyer-Peter--Muller-type closure \eqref{Meyer--Peter--Muller} \cite{Garres} used, they satisfy
	\begin{align}\label{eq:delta_def_standard}
		\delta_q&=
		\dfrac{24Q}{1-\psi}\operatorname{sgn}(u_b)
		\dfrac{\rho\epsilon}{g(\rho_s-\rho_w)d_s}(\theta-\theta_c)_+^{1/2}\dfrac{u_b}{h},\\
		\delta_h&=-u_b\left(1+\dfrac{c_m(\rho_s-\rho_w)}{2\rho}\right)\delta_q,
		\qquad
		\delta_c=u_b\left(\dfrac{\rho_s-\rho_w}{2\rho}\right)\delta_q,
	\end{align}
	where \(\rho\) is the depth-averaged mixture density defined in \eqref{eq:averaged_density}, \(\epsilon\) is the bottom-friction coefficient, \(\theta_c\) is the critical Shields parameter, and \((\cdot)_+\) denotes the positive part; see Section~\ref{bedflux} for the complete expression.
	In the SWEMED1 formulation \eqref{eq:standard_hswemed1}, the suspended concentration is transported with the depth-averaged velocity $u_m$, as in the classical depth-averaged suspended-load closure.\\
	The source term $S(W)$ of SWEMED1 \eqref{eq:standard_hswemed1} is
	\begin{equation}\label{eq:S_standard_N1}
		S(W)=
		\begin{pmatrix}
			\dfrac{E-D}{1-\psi}\\[0.5mm]
			-\epsilon |u_b|u_b+\dfrac{E-D}{1-\psi}u_b\\[0.5mm]
			-3\left(\epsilon |u_b|u_b+4\dfrac{\nu}{h}\alpha_1 \right)
			+2\dfrac{E-D}{1-\psi}\alpha_1\\[0.5mm]
			E-D\\[0.5mm]
			\dfrac{D-E}{1-\psi}
		\end{pmatrix}.
	\end{equation}
    Here, the source term \(S(W)\) contains the contributions from entrainment-deposition exchange, bottom friction, and viscous relaxation of the first velocity-moment coefficient.\\
	In the next section, we derive the equilibrium manifold of SWEMED1 \eqref{eq:standard_hswemed1} and analyze its stability.
	\section{Equilibrium manifold and stability analysis of SWEMED1}\label{sec:water_rest_stability}
	
	For the hyperbolic balance law \eqref{eq:standard_hswemed1}, the source equilibrium manifold is defined by
	\begin{equation}
		\mathcal E=\{W\in G:S(W)=0\},
	\end{equation}
	where \(G\subset\mathbb R^5\) denotes the admissible set of states for \(W\) in \eqref{eq:W_standard}. Thus \(\mathcal E\) is the set of admissible states for which the right-hand-side source forcing terms $S(W)$ \eqref{eq:S_standard_N1} vanish.
	We first derive the fully-settled water-at-rest equilibrium manifold of SWEMED1 \eqref{eq:standard_hswemed1} and then analyze its stability properties.
	\subsection{Fully-settled water-at-rest equilibrium manifold}
	\begin{theorem}\label{thm:water_rest_N1}
		For the SWEMED1 system \eqref{eq:standard_hswemed1} with \(\epsilon>0\) and \(\nu>0\), the fully-settled water-at-rest equilibrium manifold is
		\begin{equation}\label{water_at_rest}
			\Ewr=\{W:\ u_m=0,\ \alpha_1=0,\ c_m=0\}.
		\end{equation}
	\end{theorem}
	
	\begin{proof}
		We solve \(S(W)=0\) using \eqref{eq:S_standard_N1}. From the fourth row of \(S(W)\), we obtain \(E=D\). Substituting this relation into the second row gives
		$-\epsilon |u_b|u_b=0.$
		Since \(\epsilon>0\), we obtain \(u_b=0\). Using \(E=D\) and \(u_b=0\), the third row reduces to
		$-12\dfrac{\nu}{h}\alpha_1=0.$ Since \(\nu>0\) and \(h>0\), we get \(\alpha_1=0\). Consequently, from $u_b=u_m+\alpha_1=0,$
		we also obtain \(u_m=0\). Finally, \(u_b=0\) implies \(E=0\) for the entrainment closure \eqref{eq:ED_standard}, and together with \(E=D\) this gives \(D=0\). Since \(D=\omega_0S_bc_m\) with \(\omega_0>0\) and \(S_b>0\) \eqref{bed_concentration}, we conclude that \(c_m=0\).
	\end{proof}
	
	Here ``fully-settled'' refers to the condition \(c_m=0\) and ``water-at-rest'' refers to the condition $u_m = 0$ and $\alpha_1 = 0$, effectively leading to a zero velocity profile. Note that suspended water-at-rest states with \(u_m=0\) and \(\alpha_1=0\), but with \(c_m>0\), are not exact equilibria of the full source term. They can, however, describe an intermediate stage (so-called secular equilibrium \cite{Prince1979SecularEquilibrium}) of the relaxation process when the hydrodynamic variables relax faster than the suspended concentration decays through deposition, as observed numerically in Section~\ref{sec:numerical_relaxation} and investigated in Section~\ref{sec:fast_manifold}.
	
	\subsection{Yong structural stability}
	Any homogeneous equilibrium state \(W\in\mathcal E\) can be viewed as a constant solution of \eqref{eq:standard_hswemed1}. Stability of the system depends not only on the source term, but also on the interaction between the source and the transport matrix. We use Yong's structural stability framework \cite{huang2022equilibrium,Yong1999SingularPerturbations}. Writing the transport matrix as \(A(W)\) and denoting the source Jacobian by \(S_W(W)\), Yong's three stability conditions are
	\begin{align*}
	   &\text{(I) Block condition:}\, P(W)S_W(W)P(W)^{-1}=
		\begin{pmatrix}0&0\\0&\widehat T(W)\end{pmatrix},\\
		&\text{(II) Transport symmetrization:}\, A_0(W)A(W)=A(W)^TA_0(W),\\
	   &\text{(III) Dissipation compatibility:}\, A_0(W)S_W(W)+S_W(W)^TA_0(W)
		\preceq -P(W)^T
		\begin{pmatrix}0&0\\0&I_r\end{pmatrix}P(W),
	\end{align*}
	where \(P(W)\) is invertible, \(\widehat T(W)\) is invertible, and \(A_0(W)\) is symmetric positive definite.
	\begin{remark}
		For the one-dimensional case, condition \textup{(II)} can be checked through the eigenstructure of the transport matrix \(A(W)\). If \(A(W)\) has a complete set of real left eigenvectors, then a symmetrizer can be written as $A_0=L^T\Omega L,$
		where the rows of \(L\) are left eigenvectors and \(\Omega\) is a positive diagonal matrix \cite{huang2022equilibrium}. Note that Yong's conditions are sufficient stability conditions, not necessary.
	\end{remark}
	
	\begin{theorem}\label{thm:yong_standard_N1}
		For the SWEMED1 system \eqref{eq:standard_hswemed1} at the fully-settled water-at-rest manifold \eqref{water_at_rest}, Yong's condition \textup{(I)} holds, whereas conditions \textup{(II)} and \textup{(III)} fail.
	\end{theorem}
	
	\begin{proof}
		\emph{Condition \textup{(I)}.\,Block condition:}
		Consider a state $W=(h,0,0,0,h_b)^T\in\Ewr$ \eqref{water_at_rest}.
		At this state, the entrainment closure \eqref{eq:ED_standard} gives \(E=0\) and \(E_W=0\). The derivative of the quadratic bottom-friction term \(\epsilon |u_b|u_b\) also vanishes at \(u_b=0\). Therefore the source Jacobian \(S_W(W)\) is
		\begin{equation}\label{eq:SW_standard_rest}
			S_W(W)=
			\begin{pmatrix}
				0&0&0&-\dfrac{\omega_0S_b}{(1-\psi)h}&0\\[1mm]
				0&0&0&0&0\\[0.5mm]
				0&0&-\dfrac{12\nu}{h^2}&0&0\\[0.5mm]
				0&0&0&-\dfrac{\omega_0S_b}{h}&0\\[0.5mm]
				0&0&0&\dfrac{\omega_0S_b}{(1-\psi)h}&0
			\end{pmatrix}.
		\end{equation}
		To verify condition \textup{(I)}, introduce the source-adapted variables
		\begin{equation}\label{eq:P_standard_N1}
        	Y=PW=
        	\left(
        	h-\dfrac{hc_m}{1-\psi},\
        	hu_m,\
        	h_b+\dfrac{hc_m}{1-\psi},\
        	h\alpha_1,\
        	hc_m
        	\right)^T .
        \end{equation}
    	The corresponding transformation matrices \(P\) and \(P^{-1}\) are
		\begin{equation}\label{eq:P_matrix_standard_N1}
			P=
			\begin{pmatrix}
				1&0&0&-\dfrac{1}{1-\psi}&0\\[0.5mm]
				0&1&0&0&0\\[0.5mm]
				0&0&0&\dfrac{1}{1-\psi}&1\\[0.5mm]
				0&0&1&0&0\\[0.5mm]
				0&0&0&1&0
			\end{pmatrix},
			\qquad
			P^{-1}=
			\begin{pmatrix}
				1&0&0&0&\dfrac{1}{1-\psi}\\[0.5mm]
				0&1&0&0&0\\[0.5mm]
				0&0&0&1&0\\[0.5mm]
				0&0&0&0&1\\[0.5mm]
				0&0&1&0&-\dfrac{1}{1-\psi}
			\end{pmatrix}.
		\end{equation}
		A direct calculation gives
		\begin{equation}\label{eq:block_standard_conditionI}
			P S_W(W)P^{-1}=
			\begin{pmatrix}
				0_{3\times3}&0\\
				0&\widehat T
			\end{pmatrix},
			\qquad
			\widehat T=
			\begin{pmatrix}
				-\dfrac{12\nu}{h^2}&0\\
				0&-\dfrac{\omega_0S_b}{h}
			\end{pmatrix}.
		\end{equation}
		Since \(\nu>0\), \(h>0\), \(\omega_0>0\), and \(S_b>0\), the block \(\widehat T\) is invertible. Hence Yong's condition \textup{(I)} holds.\\
		\emph{Condition \textup{(II)}.\, Transport symmetrization:}
		We now examine the hyperbolicity of the transport matrix \eqref{eq:A_standard_N1} at the fully-settled water-at-rest equilibrium manifold \eqref{water_at_rest}. Substituting \(u_m=0\), \(\alpha_1=0\), and \(c_m=0\) into \eqref{eq:A_standard_N1} gives
		\begin{equation}\label{eq:A_standard_rest}
			A=
			\begin{pmatrix}
				0&1&0&0&0\\
				gh&0&0&\dfrac{g h(\rho_s-\rho_w)}{2\rho_w}&gh\\[2mm]
				0&0&0&\dfrac{g h(\rho_s-\rho_w)}{2\rho_w}&0\\[2mm]
				0&0&0&0&0\\
				0&0&0&0&0
			\end{pmatrix}.
		\end{equation}
		The characteristic polynomial is
		\begin{equation}
			\det(\lambda I-A)=\lambda^3(\lambda^2-gh).
		\end{equation}
		The simple eigenvalues are \(\lambda_\pm=\pm\sqrt{gh}\), with eigenvectors
		\begin{equation}\label{eq:standard_rest_gravity_eigenvectors}
			r_-=(1,-\sqrt{gh},0,0,0)^T,
				\qquad
				r_+=(1,\sqrt{gh},0,0,0)^T.
		\end{equation}
		The remaining eigenvalue \(\lambda=0\) has algebraic multiplicity three, but only two linearly independent eigenvectors,
		\begin{equation}\label{eq:standard_rest_eigenvectors}
			r_0^{(1)}=(1,0,0,0,-1)^T,
				\qquad
				r_0^{(2)}=(0,0,1,0,0)^T.
		\end{equation}
		Thus the transport matrix \(A(W)\) in \eqref{eq:A_standard_rest} is not diagonalizable and is only weakly hyperbolic at the fully-settled water-at-rest equilibrium. Since condition \textup{(II)} requires a positive definite symmetrizer for the transport matrix, and such a symmetrizer cannot be constructed for a weakly hyperbolic matrix, Yong's condition \textup{(II)} fails. Consequently, the full set of Yong structural stability conditions cannot hold, and condition \textup{(III)} also fails.
	\end{proof}
	%\begin{remark}
		The failure of Yong's structural stability conditions for SWEMED1 \eqref{eq:standard_hswemed1} at the fully-settled water-at-rest manifold \eqref{water_at_rest} should be understood as a structural obstruction associated with the strict manifold, not as a defect of the SWEMED1 model. At this manifold, the full equilibrium condition imposes not only \(u_m=0\) and \(\alpha_1=0\), but also \(c_m=0\). In this limiting configuration, the coupling between the transport matrix and the source dissipation does not fit the block structure required by Yong's framework. Since Yong's conditions are sufficient but not necessary for stability, their failure does not imply instability of SWEMED1. It only shows that this particular structural stability framework is too restrictive for SWEMED1 at the strict fully-settled water-at-rest manifold. We therefore continue investigating the linear spectral stability in the next section.
	%\end{remark}
	
	\subsection{Linear spectral stability}
	
	To complement the structural stability condition result above, we examine linear spectral stability of SWEMED1 \eqref{eq:standard_hswemed1}. Here linear spectral stability is understood in the non-growing sense: after linearization, every Fourier mode has a growth rate with non-positive real part \cite{huang2022equilibrium}. Neutral modes are therefore allowed in our definition, and the result does not imply asymptotic decay of all modes. More precisely, we linearize the SWEMED1 system \eqref{eq:standard_hswemed1} around a homogeneous fully-settled water-at-rest state \(W\in\Ewr\) \eqref{water_at_rest}. This gives
	\begin{equation}\label{eq:linearized_standard}
		\partial_t\delta W+A\partial_x\delta W=S_W(W)\delta W,
	\end{equation}
	where \(\delta W\) denotes the deviation from this equilibrium state and \(A\) is given in \eqref{eq:A_standard_rest} and is evaluated at the equilibrium. Substituting the normal-mode ansatz
	\begin{equation}\label{eq:normal_mode}
		\delta W(t,x)=\widehat W e^{\lambda t+i\xi x},
		\qquad \xi\in\R,
	\end{equation}
	where \(\widehat W\) is the constant Fourier-mode amplitude, \(\xi\) is the wave number, and \(\lambda\) is the spectral parameter describing the temporal growth rate, yields the algebraic eigenvalue problem
	\begin{equation}\label{eq:fourier_standard}
		\lambda\widehat W=(S_W(W)-i\xi A)\widehat W.
	\end{equation}
	Using \(S_W(W)\) in \eqref{eq:SW_standard_rest} and \(A(W)\) in \eqref{eq:A_standard_rest}, a direct calculation gives
	\begin{equation}\label{eq:spectral_polynomial_N1}
		\det\left(\lambda I-(S_W(W)-i\xi A)\right)
		=\lambda
		\left(\lambda+\dfrac{\omega_0S_b}{h}\right)
		\left(\lambda+\dfrac{12\nu}{h^2}\right)
		(\lambda^2+gh\xi^2).
	\end{equation}
	Hence
	\begin{equation}\label{eq:spectral_values_N1}
		\sigma(S_W(W)-i\xi A)=
		\left\{
		0,\ 
		-\dfrac{\omega_0S_b}{h},\ 
		-\dfrac{12\nu}{h^2},\ 
		\pm i\sqrt{gh}\,\xi
		\right\}.
	\end{equation}
	All eigenvalues have non-positive real part, i.e.,
	\begin{equation}
		Re(\lambda)\le 0
		\qquad
		\text{for every wave number } \xi\in\R.
	\end{equation}
	Thus the linearized SWEMED1 system has no growing normal modes at the fully-settled water-at-rest equilibrium. The modes \(0\) and \(\pm i\sqrt{gh}\,\xi\) are neutral: they are not damped, but they do not grow. This indicates neutral linear spectral stability of the SWEMED1.
	
	\section{Numerical simulations}\label{sec:numerical_relaxation}
	
	We now investigate the numerical relaxation behavior of the SWEMED1 system \eqref{eq:standard_hswemed1} toward the fully-settled water-at-rest manifold \eqref{water_at_rest}. The purpose is not to verify Yong's sufficient structural stability conditions numerically. Rather, the test checks whether the solution relaxes toward the fully-settled state and whether growing perturbations are observed.
	
	As relevant for the numerical solution, the SWEMED1 system \eqref{eq:standard_hswemed1} contains a non-conservative transport part and a local source part accounting for friction and bed-suspension exchange. Since the source terms may be stiff due to small $h$ and act locally in each grid cell, for the numerical scheme we use an operator-splitting strategy following \cite{huang2022equilibrium}. The transport step is discretized by a non-conservative finite-volume method with explicit RK34 time integration \cite{huang2022equilibrium}. The source step is advanced by the implicit Euler method in each cell, and the resulting nonlinear system is solved locally using Newton's method.
	
	The computation is performed on \(x\in[-1,2]\). We take a perturbed water height and nonzero initial hydrodynamic variables, as in \cite{huang2022equilibrium},
	\begin{align}\label{eq:num_initial_data}
		h(0,x)&=\begin{cases}1.5,&x<0,\\1.0,&x>0,\end{cases}
		&u_m(0,x)&=0.05,
		&\alpha_1(0,x)&=-0.01,
	\end{align}
	with
	\begin{equation}\label{eq:num_initial_data2}
		c_m(0,x)=\begin{cases}0.01,&x<0,\\0,&x>0,\end{cases}
		\qquad
		h_b(0,x)=0.
	\end{equation}
	Open boundary conditions are imposed at both ends of the computational domain. For this test, we use the bottom-friction coefficient \(\epsilon=15\) and the moment relaxation parameter \(\nu=10\) to obtain fast hydrodynamic relaxation. The snapshots are shown at different times.
	To monitor the relaxation of the hydrodynamic variables $u_m$ and $\alpha_1$ toward rest, we use the hydrodynamic deviation measure, compare \cite{huang2022equilibrium},
	\begin{equation}\label{eq:EQ1_definition}
		EQ_1=|u_m|+|\alpha_1|,
	\end{equation}
	which vanishes when the mean velocity \(u_m\) and the first velocity-moment coefficient \(\alpha_1\) vanish.
	
	Figure~\ref{fig:h_um}(a) shows the evolution of the water height \(h\). The initial discontinuity becomes smoother during the evolution. Figure~\ref{fig:h_um}(b) shows the mean velocity \(u_m\), which initially changes due to transport but then decreases toward zero as bottom friction acts. Figure~\ref{fig:h_um}(c) shows that the first velocity-moment coefficient \(\alpha_1\) also relaxes toward zero. In contrast, the suspended concentration \(c_m\) in Figure~\ref{fig:h_um}(d) decays more slowly through deposition $D$. This behavior suggests a separation between the faster hydrodynamic relaxation and the slower deposition process in this test, which motivates the fast-slow source splitting introduced in Section~\ref{sec:fast_manifold}. Figure~\ref{fig:h_um}(e) shows that the reconstructed velocity profile \(u(t,0,\zeta)\) becomes nearly flat with small magnitude, while Figure~\ref{fig:h_um}(f) confirms the decay of the hydrodynamic deviation measure \(EQ_1\) \eqref{eq:EQ1_definition}. Altogether, the numerical solution relaxes toward the fully-settled water-at-rest state and does not show growing perturbations. This behavior is consistent with the linear spectral result in Section~\ref{sec:water_rest_stability}: the SWEMED1 system fails Yong's sufficient structural conditions at the strict equilibrium, but the linearized model does not contain growing normal modes.
	
	\begin{figure}[t]
		\sidecaption[t]
		\resizebox{0.65\textwidth}{!}{%
			\begin{tabular}[b]{cc}
				\includegraphics[width=0.31\textwidth]{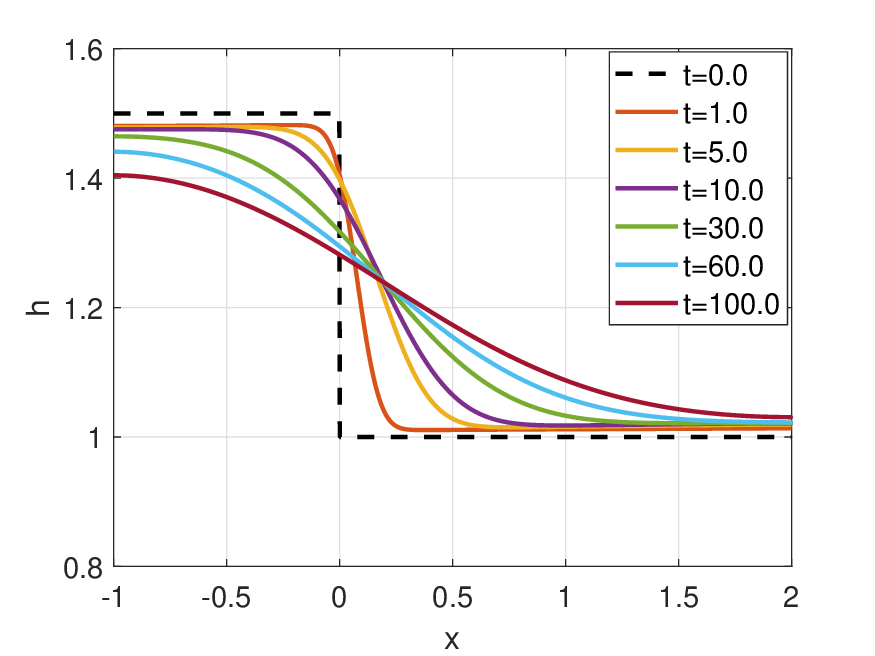} &
				\includegraphics[width=0.31\textwidth]{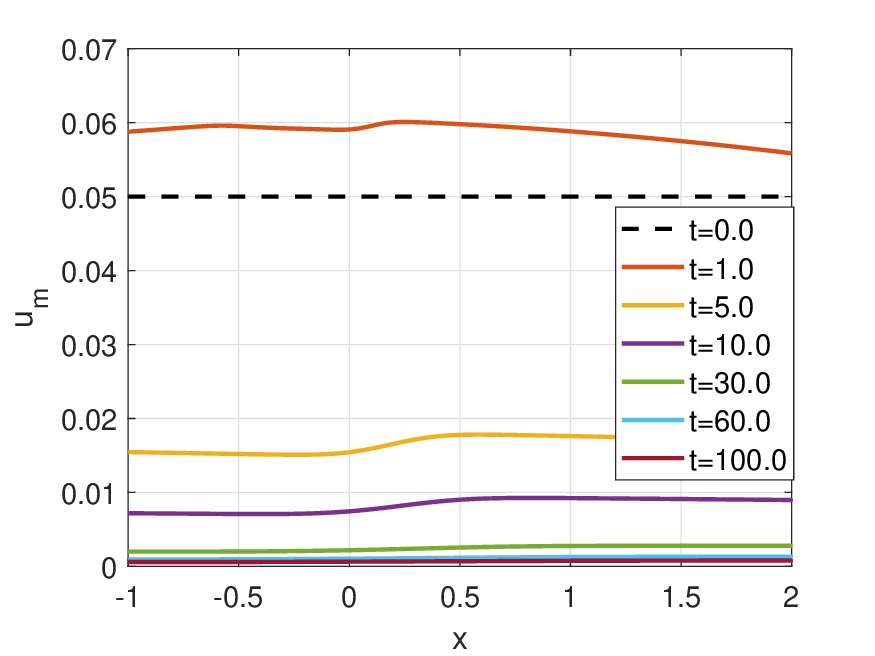} \\
				{\scriptsize \textbf{a} water height \(h\)} &
				{\scriptsize \textbf{b} mean velocity \(u_m\)} \\[0.6em]
				\includegraphics[width=0.31\textwidth]{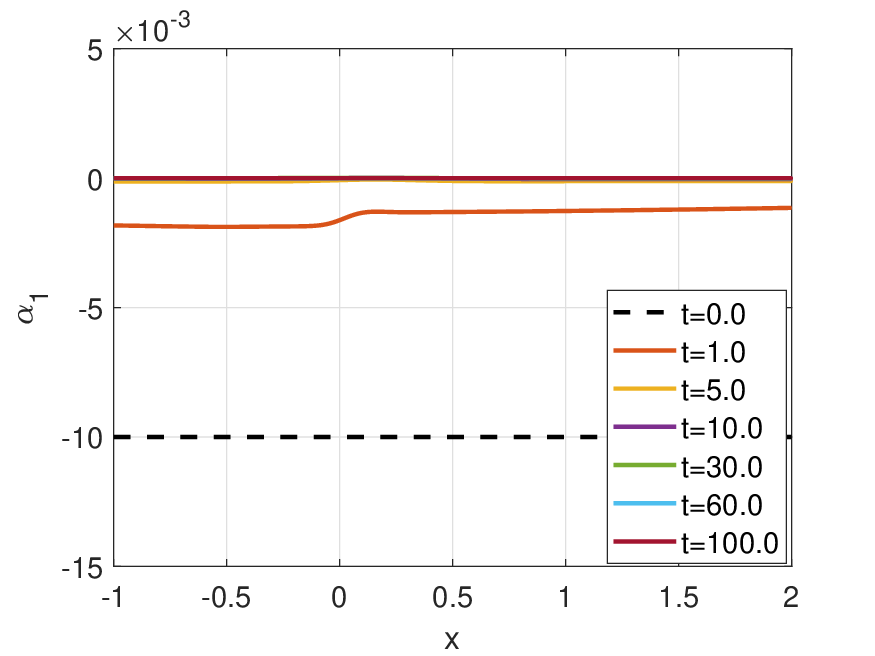} &
				\includegraphics[width=0.31\textwidth]{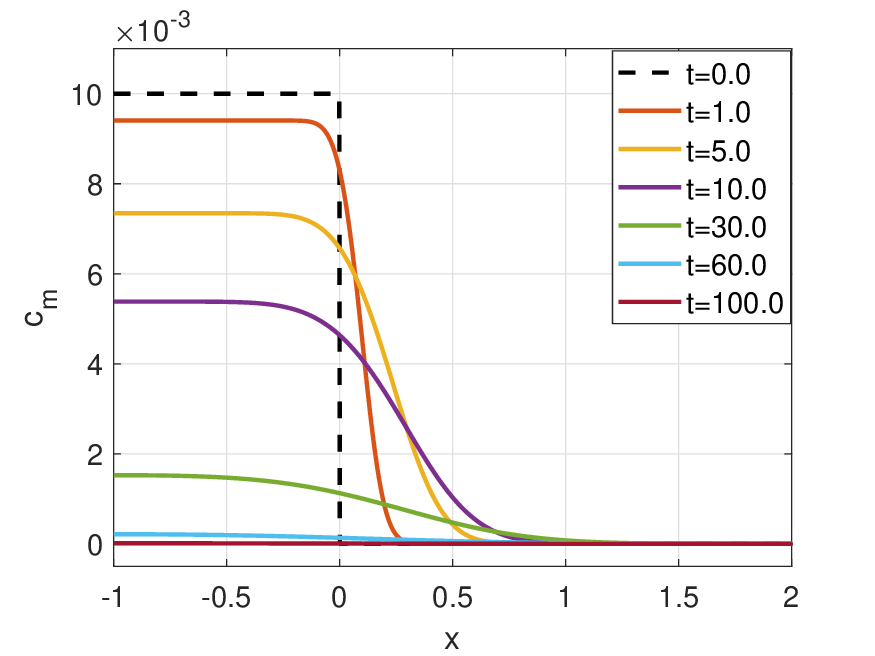} \\
				{\scriptsize \textbf{c} first velocity-moment coefficient \(\alpha_1\)} &
				{\scriptsize \textbf{d} suspended concentration \(c_m\)} \\[0.6em]
				\includegraphics[width=0.31\textwidth]{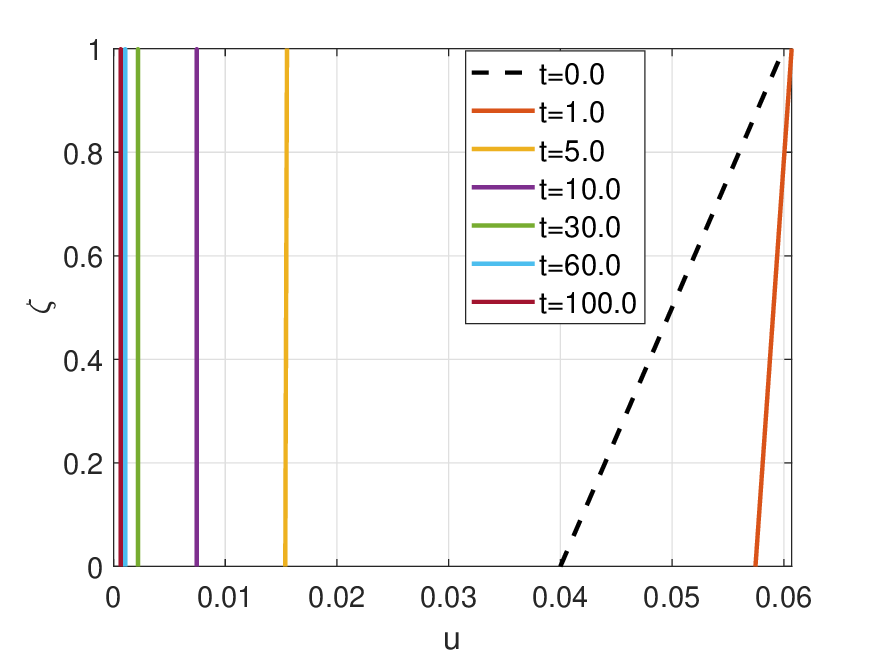} &
				\includegraphics[width=0.31\textwidth]{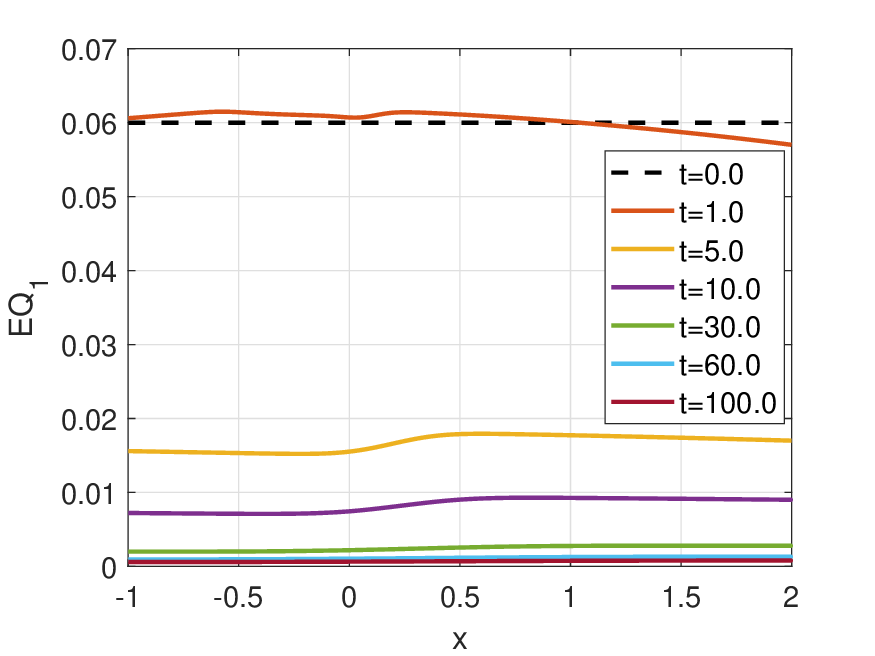} \\
				{\scriptsize \textbf{e} velocity profile \(u(\zeta)\) at \(x=0\)} &
				{\scriptsize \textbf{f} hydrodynamic deviation \(EQ_1\)}
		\end{tabular}}
		\Description{Six-panel numerical relaxation plot showing water height, mean velocity, first velocity-moment coefficient, suspended concentration, reconstructed velocity profile, and hydrodynamic deviation for the fully-settled water-at-rest test.}
		\caption{Fully-settled water-at-rest relaxation test for the SWEMED1 system \eqref{eq:standard_hswemed1} with bottom-friction coefficient $\epsilon = 15$ and moment relaxation
			parameter $\nu = 10$ at times $t = 0, 1, 5, 10, 30, 60, 100$. The hydrodynamic velocity variables $u_m$ and $\alpha_1$ relax toward rest, while the suspended concentration $c_m$ decays more slowly through deposition.}
		\label{fig:h_um}
	\end{figure}
	\section{Fast-manifold analysis for SWEMED1}\label{sec:fast_manifold}
	The fully-settled water-at-rest equilibrium \eqref{water_at_rest} imposes $c_m=0$. The numerical simulation in Section~\ref{sec:numerical_relaxation}, however, shows an intermediate regime in which $u_m$ and $\alpha_1$ are already close to zero while $c_m$ remains positive. This is sometimes referred to as secular equilibrium \cite{Prince1979SecularEquilibrium}. To describe this stage analytically, we introduce the fast-slow source splitting
	\begin{equation}\label{eq:source_splitting}
		S_\delta(W)=S^{\mathrm{fast}}(W)+\delta \cdot S^{\mathrm{slow}}(W),
		\qquad 0<\delta\ll1.
	\end{equation}
	The fast source \(S^{\mathrm{fast}}(W)\) contains bottom friction, moment relaxation, and entrainment, whereas the slow source \(S^{\mathrm{slow}}(W)\) contains deposition. Thus
	\begin{equation}\label{eq:fast_slow_sources}
	\begin{aligned}
	S^{\mathrm{fast}}(W)&=\begin{pmatrix}
	\dfrac{E}{1-\psi}\\[1mm]
	-\mu u_b+\dfrac{E}{1-\psi}u_b\\[1mm]
	-3\left(\mu u_b+4\dfrac{\nu}{h}\alpha_1\right)+2\dfrac{E}{1-\psi}\alpha_1\\[1mm]
	E\\[1mm]
	-\dfrac{E}{1-\psi}
	\end{pmatrix},
	&
	S^{\mathrm{slow}}(W)&=\begin{pmatrix}
	-\dfrac{D}{1-\psi}\\[1mm]
	-\dfrac{D}{1-\psi}u_b\\[1mm]
	-2\dfrac{D}{1-\psi}\alpha_1\\[1mm]
	-D\\[1mm]
	\dfrac{D}{1-\psi}
	\end{pmatrix}.
	\end{aligned}
	\end{equation}
    
	Here $\mu>0$ denotes the coefficient in the local linear friction approximation used in the fast structural calculation. For this analysis, the quadratic bottom-friction term \(\epsilon |u_b|u_b\) used in \eqref{depth_avg_momentum} and \eqref{final_higher_average_equation}, is replaced by the linear damping term \(\mu u_b\).
    This keeps the fast friction contribution active on the fast manifold \eqref{eq:Mfast}.  In the limit $\delta\to0$, the leading-order source equilibrium is determined by $S^{\mathrm{fast}}(W)=0$, which we investigate below.
    \begin{theorem}\label{thm:fast_manifold_source}
		In the fast limit \(\delta\to 0\) of \eqref{eq:source_splitting}, the fast source equilibrium manifold is the suspended water-at-rest.
		\begin{equation}\label{eq:Mfast}
			\Mfast=\{W\in G: u_m=0,\ \alpha_1=0,\ c_m > 0\}.
		\end{equation}
	\end{theorem}
	\begin{proof}
		From the fourth row of $S^{\mathrm{fast}}(W)=0$, we obtain $E=0$. Substituting this into the second row gives $u_b=0$. The third row then reduces to $-12(\nu/h)\alpha_1=0$, so that $\alpha_1=0$ since $\nu>0$ and $h>0$. Hence $u_m=0$ follows from $u_b=u_m+\alpha_1=0$. Since deposition belongs to the slow source, $c_m$ is not forced to vanish at the fast time scale.
	\end{proof}
    The manifold \eqref{eq:Mfast} is the leading-order source equilibrium obtained in the fast limit \(\delta\to 0\), not the full source equilibrium of SWEMED1. It separates the fast hydrodynamic relaxation from the slow deposition process and therefore gives a description of intermediate states where the hydrodynamic variables have reached rest while suspended sediment is still present.

	We now check whether this fast-manifold \eqref{eq:Mfast} viewpoint alone removes the transport
    obstruction found at the fully-settled water-at-rest manifold \eqref{water_at_rest}. The answer is negative if the SWEMED1 concentration
    equation retains the depth-averaged transport velocity \(u_m\).

\begin{theorem}\label{thm:fast_manifold_failure}
	For SWEMED1 \eqref{eq:standard_hswemed1} at the fast suspended water-at-rest manifold
	\eqref{eq:Mfast}, obtained from the fast limit \(\delta\to0\) of
	\eqref{eq:source_splitting}, Yong's condition \textup{(I)} holds for
	the fast source, whereas condition \textup{(II)} fails. Consequently,
	the full set of Yong structural stability conditions cannot be verified
	at this manifold.
\end{theorem}

\begin{proof}
	\emph{Condition \textup{(I)}.\,Block condition:}
	We use the reordered variables
	\begin{equation}\label{eq:Y_fast_ordering}
	   Y=(h,hc_m,h_b,hu_m,h\alpha_1)^T .
    \end{equation}
	At a state \(W\in\Mfast\) defined in \eqref{eq:Mfast}, and using the fast source from the splitting \eqref{eq:source_splitting}, the source
	Jacobian with respect to the variables \(Y\) in \eqref{eq:Y_fast_ordering}
	has the block form
	\begin{equation}\label{eq:SY_fast_standard_block}
		S^{\mathrm{fast}}_Y(Y)=
		\begin{pmatrix}
			0_{3\times3}&0\\
			0&B
		\end{pmatrix},
		\qquad
		B=\begin{pmatrix}
			-\dfrac{\mu}{h}&-\dfrac{\mu}{h}\\[2mm]
			-\dfrac{3\mu}{h}&-\left(\dfrac{3\mu}{h}+\dfrac{12\nu}{h^2}\right)
		\end{pmatrix}.
	\end{equation}
	Since \(h>0\), \(\mu>0\), and \(\nu>0\), the block \(B\) is invertible.
	Hence Yong's condition \textup{(I)} holds for the fast source.\\

	\emph{Condition \textup{(II)}.\,Transport symmetrization:}
	We now examine the hyperbolicity of the transport matrix \(A(W)\) of SWEMED1 \eqref{eq:A_standard_N1} at the fast suspended water-at-rest manifold \eqref{eq:Mfast}. With respect to the conservative variables \(W=(h,hu_m,h\alpha_1,hc_m,h_b)^T\) in \eqref{eq:W_standard}, evaluating \eqref{eq:A_standard_N1} at \(u_m=0\), \(\alpha_1=0\), \(c_m>0\), and hence \(u_b=0\), gives
	
	\begin{equation}\label{eq:AW_fast_standard}
	A_{\Mfast}(W)=
	\begin{pmatrix}
	0&1&0&0&0\\
	gh-c_m\beta&0&0&\beta&gh\\
	-c_m\beta&0&0&\beta&0\\
	0&c_m&0&0&0\\
	0&0&0&0&0
	\end{pmatrix}, \qquad \beta=\dfrac{gh(\rho_s-\rho_w)}{2\rho}.
	\end{equation}
   Using the same reordered variables \(Y=(h,hc_m,h_b,hu_m,h\alpha_1)^T\) defined in \eqref{eq:Y_fast_ordering}, the transport matrix \eqref{eq:AW_fast_standard} becomes
	\begin{equation}\label{eq:AY_fast_standard}
	A^Y_{\Mfast}=
	\begin{pmatrix}
	0&0&0&1&0\\
	0&0&0&c_m&0\\
	0&0&0&0&0\\
	gh-c_m\beta&\beta&gh&0&0\\
	-c_m\beta&\beta&0&0&0
	\end{pmatrix}.
	\end{equation}
    A direct calculation from
	\eqref{eq:AY_fast_standard} gives
	\begin{equation}\label{eq:char_fast_standard}
		\det(\lambda I-A^Y_{\Mfast})=\lambda^3(\lambda^2-gh).
	\end{equation}
	Thus, the zero eigenvalue has algebraic multiplicity three. For
	\(\lambda=0\), let \(r=(r_1,r_2,r_3,r_4,r_5)^T\) be a right eigenvector,
	with components ordered according to \(Y=(h,hc_m,h_b,hu_m,h\alpha_1)^T\) in
	\eqref{eq:Y_fast_ordering}. The equation \(A^Y_{\Mfast}r=0\) gives
	\(r_4=0\), \(c_m r_4=0\), and
	\begin{equation*}
		gh r_1-\beta r_2+gh r_3=0,
		\qquad -\beta r_2=0.
	\end{equation*}
	Since \(\beta=\dfrac{gh(\rho_s-\rho_w)}{2\rho}>0\), we obtain
	\(r_2=0\) and \(r_1=-r_3\), while \(r_5\) is free. Thus, the kernel is
	spanned, for instance, by
	\begin{equation*}
		(1,0,-1,0,0)^T,
		\qquad
		(0,0,0,0,1)^T .
	\end{equation*}
	The zero eigenvalue, therefore, has geometric multiplicity two, while its algebraic multiplicity is three. Hence \(A^Y_{\Mfast}\) in
	\eqref{eq:AY_fast_standard} is not diagonalizable and is only weakly
	hyperbolic at the fast suspended water-at-rest manifold. Since Yong's
	condition \textup{(II)} requires a positive definite symmetrizer for the
	transport matrix, condition \textup{(II)} fails. Consequently, the full
	set of Yong's structural stability conditions cannot be verified at the fast
	suspended water-at-rest manifold.
\end{proof}
    
    The fast-slow splitting, therefore, changes the source equilibrium structure but does not, by itself, repair the transport degeneracy. The remaining obstruction is tied to the suspended-concentration transport closure: when the concentration is transported with $u_m$, which vanishes at water-at-rest equilibrium, the concentration row does not provide an additional independent transport coupling at water-at-rest. This suggests that a future extension should examine effective transport velocities for suspended sediment, derived consistently with the reduced modelling assumptions. One way to do so would be to replace the $u_m$-based concentration flux by a flux derived from an effective sediment transport velocity. The modelling and stability analysis of such closures is left for future work.
    
	\section{Conclusions}\label{sec:conclusion}
	In this work, we derived and analyzed SWEMED1, a first-moment shallow water Exner moment model with sediment entrainment and deposition. The model couples the water mass balance, depth-averaged momentum balance, first velocity-moment equation, suspended concentration equation, and Exner bed-evolution equation. The suspended concentration is transported as a depth-averaged scalar and feeds back into the hydrodynamics through the depth-averaged mixture density.
    
    An analysis of the source term showed that for non-vanishing bottom friction, moment relaxation, settling velocity, and near-bed concentration factor, the source term forces the mean velocity, the first velocity-moment coefficient, and the suspended concentration to vanish in equilibrium. This results in the fully-settled water-at-rest equilibrium manifold. We then examined its stability using Yong's structural stability framework. The source Jacobian satisfies Yong's block condition, but the transport matrix at the fully-settled water-at-rest manifold is only weakly hyperbolic. %More precisely, the zero eigenvalue has algebraic multiplicity three but geometric multiplicity two. 
    Therefore, Yong's transport symmetrizer condition cannot be verified, and the full set of Yong structural stability conditions cannot be established at this manifold. However, this obstruction should not be interpreted as instability of SWEMED1, since Yong's conditions are sufficient but not necessary.

    To complement the structural result, we performed a linear spectral stability calculation around a homogeneous fully-settled water-at-rest state. The resulting spectrum contains no eigenvalue with a positive real part. The linearized system thus has no growing normal modes in this setting. A numerical test is consistent with this non-growing behavior and in addition, shows that the hydrodynamic variables approach rest while the suspended concentration may decay on a slower time scale.

    Motivated by this observation, we then introduced a fast-slow source splitting and considered the fast limit. The corresponding fast equilibrium is a suspended water-at-rest manifold with non-vanishing concentration. This manifold has a different source-equilibrium structure compared to the fully-settled water-at-rest manifold. However, when the suspended concentration is still transported with the depth-averaged velocity, the transport matrix at the fast manifold remains weakly hyperbolic. Thus, the fast-slow splitting by itself does not remove the transport obstruction. These results show that the remaining obstruction is linked to the transport closure in the suspended-concentration equation. This suggests that future extensions should examine effective transport velocities for suspended sediment, derived consistently with the reduced modelling assumptions and analyzed together with the corresponding equilibrium and stability structure.

	\section*{Acknowledgements}
		The authors gratefully acknowledge fruitful discussions with Prof. Wen-An Yong. They also acknowledge financial support from the KU Leuven Global PhD Partnership fellowship for a joint PhD with Peking University (grant agreement GPPKU/21/009). This work is part of the HiWAVE project (file no. VI.Vidi.233.066) within the NWO Vidi ENW programme, partly funded by the Dutch Research Council (NWO; grant DOI: 10.61686/CBVAB59929).
		
	\ethics{Competing Interests}{The authors have no conflicts of interest to declare that are relevant to the content of this paper.}

 \section*{Appendices}
\appendix
\setcounter{section}{0}
\renewcommand{\thesection}{~\Alph{section}}
\renewcommand{\thesubsection}{\thesection.\arabic{subsection}}
\section{Derivation of the complete reference system}\label{app:1}
  
    \subsection{Mapping of the mass balance}
    From the divergence-free conditions $\nabla \cdot \mathbf{u}=0$, we have
    \begin{equation}\label{r1.1}
        \partial_x u + \partial_z w = 0.
    \end{equation}
    After multiplying \eqref{r1.1} by $h$ and applying the differential operator \eqref{diffop1}, the mapped mass balance yields
    \begin{equation}
        \partial_x \left(h \Tilde{u}\right)-\partial_\zeta \left[\partial_x(\zeta h+h_b) \Tilde{u} \right] + \partial_\zeta \Tilde{w} = 0.
    \end{equation}
    \subsection{Mapping of the momentum balance}
    From the reference system \eqref{reference system22},
    \begin{equation}\label{ref:A2}
        \partial_t u + \partial_x u^2 +\partial_z (uw) = -\dfrac{1}{\rho} \partial_x p + \dfrac{1}{\rho} \partial_z \sigma_{xz}.
    \end{equation}
    After multiplying \eqref{ref:A2} by $h$ and applying the differential operator \eqref{diffop1}, the mapped momentum balance yields 
    \begin{align}
        \begin{split}
                \partial_t (h \Tilde{u}) - \partial_\zeta \left(\Tilde{u} \partial_t\left(\zeta h+h_b\right)\right)& + \partial_x(h \Tilde{u}^2) -\partial_\zeta \left(\Tilde{u}^2 \partial_x\left(\zeta h+h_b\right)\right) + \partial_\zeta(\Tilde{u}\Tilde{w}) \\
                &+ \dfrac{1}{\rho} \partial_x (h \Tilde{p})- \dfrac{1}{\rho} \partial_\zeta \left(\Tilde{p}\partial_x(\zeta h+h_b)\right) = \dfrac{1}{\rho}\partial_\zeta \Tilde{\sigma}_{xz},
        \end{split}
        \\[2ex]
           \begin{split}\label{ref:A2_p}
              \Longrightarrow  \partial_t (h \Tilde{u})+ \partial_x(h \Tilde{u}^2) + & \partial_\zeta \left[\Tilde{u}\left( \underbrace{\Tilde{w} -\Tilde{u}\partial_x\left(\zeta h+h_b\right)- \partial_t\left(\zeta h+h_b\right)}_{\text{vertical coupling}}\right)\right] \\
                &+ \underbrace{ \dfrac{1}{\rho} \partial_x (h \Tilde{p})- \dfrac{1}{\rho} \partial_\zeta \left(\Tilde{p}\partial_x(\zeta h+h_b)\right)}_{\text{pressure}} = \dfrac{1}{\rho}\partial_\zeta \Tilde{\sigma}_{xz}.
        \end{split}
    \end{align}
    From the hydrostatic balance, we have $p(t,x,z)= (h_s-z)\rho g$. After mapping from $z$ to $\zeta$ we get
    \begin{align}
        p(t,x,\zeta)  = (h_s-(\zeta h+ h_b)) \rho g
                      = h (1 -\zeta)\rho g.
    \end{align}
    Therefore, 
	    \begin{align}\label{pressure_der}
	    &\dfrac{1}{\rho}\partial_x(h\Tilde p)
	    -\dfrac{1}{\rho}\partial_\zeta\!\left(\Tilde p\,\partial_x(\zeta h+h_b)\right)\notag\\
	    &\quad=\dfrac{1}{\rho}\partial_x\!\left(\rho g h^2(1-\zeta)\right)
	    -\dfrac{1}{\rho}\partial_\zeta\!\left(\rho gh(1-\zeta)\partial_x(\zeta h+h_b)\right)\notag\\
	    &\quad=2gh(1-\zeta)\partial_x h
	    +\dfrac{gh^2}{\rho}(1-\zeta)\partial_x\rho
	    -gh\partial_xh+2gh\zeta\partial_xh+gh\partial_xh_b\notag\\
	    &\quad=gh\partial_x(h+h_b)
	    +\dfrac{gh^2}{\rho}(1-\zeta)\partial_x\rho .
	    \end{align}
    Now we substitute \eqref{pressure_der} to \eqref{ref:A2_p}
	    \begin{multline}\label{ref:A2_subp}
	    \partial_t(h\Tilde u)+\partial_x(h\Tilde u^2)
	    +\partial_\zeta\!\left[\Tilde u\left(\Tilde w
	    -\Tilde u\partial_x(\zeta h+h_b)-\partial_t(\zeta h+h_b)\right)\right]\\
	    +gh\partial_x(h+h_b)+\dfrac{gh^2}{\rho}(1-\zeta)\partial_x\rho
	    =\dfrac{1}{\rho}\partial_\zeta\Tilde\sigma_{xz}.
	    \end{multline}
    
    \subsection{Mapping of the sediment concentration equation}
    From the reference system \eqref{reference system23}, we have
    \begin{equation}\label{ref:A3}
        \partial_t c +\partial_x (cu) + \partial_z (cw) = 0.
    \end{equation}
    After multiplying \eqref{ref:A3} by $h$ and applying the differential operator \eqref{diffop1}, the mapped sediment concentration equation yields
    \begin{align}
     \partial_t (h \Tilde{c})  + \partial_x \left(h \Tilde{c}\Tilde{u}\right) +\partial_\zeta \left[\Tilde{c} \left(\Tilde{w}- \partial_x(\zeta h+h_b)\Tilde{u}-\partial_t(\zeta h+h_b) \right) \right]=0.
    \end{align}

    \section{Averages of the Horizontal Momentum Balances}\label{app:2}
    \subsection{Depth-averaging the momentum balance}\label{averaged momentum balance}
    We consider the momentum balance \eqref{eq:resolved  system1_2} and after integrating $\int_{0}^{1}\cdot \, d\zeta$, we get
	    \begin{align}\label{ref:A2_integration}
	    &\partial_t\left(h\int_0^1\Tilde u\,d\zeta\right)
	    +\partial_x\left(h\int_0^1\Tilde u^2\,d\zeta\right)
	    +gh\partial_xh+gh\partial_xh_b
	    +\dfrac{gh^2}{2\rho}\partial_x\rho\notag\\
	    &\qquad=\dfrac{1}{\rho}\int_0^1\partial_\zeta\Tilde\sigma_{xz}\,d\zeta .
	    \end{align}
    After applying the kinematic boundary conditions \eqref{bc_zeta_1}-\eqref{bc_zeta_2} and substituting $\Tilde{u}(t,x,\zeta)=u_m(t,x)+\alpha_1(t,x)\phi_1(\zeta)$, we get
    \begin{align}   
    	\begin{split}
    		\partial_t(h u_m) +  \partial_x \left( hu_m^2+ h\dfrac{\alpha_1^2}{3}+  \dfrac{gh^2}{2} \right)
    		=  - gh \partial_x h_b  - \dfrac{gh^2}{2 \rho}(\rho_s-\rho_w)\partial_x c_m +  F_b u_b - \epsilon |u_b|u_b,
    	\end{split}
    \end{align}
   which represents the depth-averaged momentum balance \eqref{depth_avg_momentum}.
    %%%%%%%%%%%%%%  First order averages %%%%%%%%%%%%
    
    \subsection{First-order projection}\label{Higher order averages}
    To get the first-order projection of the momentum equation, multiply \eqref{eq:resolved  system1_2} with the test function $\phi_1(\zeta)$ and integrate from $\zeta=0$ to $\zeta=1$, 
	    \begin{align}
	    &\partial_t\left(h\int_0^1\phi_1\Tilde u\,d\zeta\right)
	    +\partial_x\left(h\int_0^1\phi_1\Tilde u^2\,d\zeta\right)
	    +gh\partial_x(h+h_b)\int_0^1\phi_1\,d\zeta\notag\\
	    &\quad+\int_0^1\phi_1\partial_\zeta\!\left[\Tilde u\left(\Tilde w
	    -\Tilde u\partial_x(\zeta h+h_b)-\partial_t(\zeta h+h_b)\right)\right]d\zeta\notag\\
	    &=-\dfrac{gh^2}{\rho}\partial_x\rho\int_0^1(1-\zeta)\phi_1\,d\zeta
	    +\dfrac{1}{\rho}\int_0^1\phi_1\partial_\zeta\Tilde\sigma_{xz}\,d\zeta .
	    \end{align}
    We have, 
    \begin{align}
    	\int_{0}^{1} \phi_1 \Tilde{u} \, d\zeta &=\int_{0}^{1} \phi_1 \left(u_m+  \alpha_1\phi_1\right) \, d\zeta
    	= \dfrac{\alpha_1}{3}.
    \end{align}
    and
    \begin{align}
    	\int_{0}^{1} \phi_1 \Tilde{u}^2 \,  \, d\zeta &=\int_{0}^{1} \phi_1 \left(u_m^2 + 2 u_m  \alpha_1\phi_1 + \alpha_1^2\phi_1^2 \right) \, d\zeta= \dfrac{2}{3}u_m\alpha_1.
    \end{align}
	    The projection on the vertical-coupling term reads
	    \begin{align}
	    &\int_0^1\phi_1\partial_\zeta\!\left[\Tilde u\left(\Tilde w
	    -\Tilde u\partial_x(\zeta h+h_b)-\partial_t(\zeta h+h_b)\right)\right]d\zeta\notag\\
	    &=-\int_0^1\phi_1\partial_\zeta\!\left[\Tilde u\,
	    \partial_x\left(h\int_0^\zeta \Tilde u\,d\hat\zeta\right)\right]d\zeta
	    -F_b\int_0^1\phi_1\partial_\zeta(\Tilde u\zeta)\,d\zeta\notag\\
	    &\qquad-\partial_t h_b\int_0^1\phi_1\partial_\zeta\Tilde u\,d\zeta .
	    \end{align}
    
    We notice the presence of the integral associated with entrainment and deposition 
    \begin{align}
    	& F_b\int_{0}^{1} \phi_1 \partial_\zeta \left(\Tilde{u}\zeta\right) \, d\zeta
    	=  F_b\left(  \int_{0}^{1} \alpha_1 \phi_1^2 \, d\zeta +  \int_{0}^{1} \alpha_1 \zeta \phi_1 \phi_1^{\prime}\, d\zeta \right)
    	=  \dfrac{2F_b\alpha_1}{3}.
    \end{align}
    and the projection related to the coupling with the bed evolution 
    \begin{align}\label{A_s defination}
    	\partial_t h_b \int_{0}^{1} \phi_1 \partial_\zeta \Tilde{u} \, d\zeta  
    	=  \partial_t h_b \alpha_1 \left(\int_{0}^{1}   \phi_1 \phi_1^{\prime} \, d\zeta \right) 
    	=0. 
    \end{align}
    Therefore, the projection of the vertical coupling for first-order reads
    \begin{align}
    \int_{0}^{1} \phi_1 \partial_\zeta &\left[\Tilde{u}\left(\Tilde{w} -\Tilde{u}\partial_x\left(\zeta h+h_b\right)- \partial_t\left(\zeta h+h_b\right)\right)\right] \, d\zeta \notag
    = - \dfrac{u_m}{3}\partial_x (h \alpha_1) - \dfrac{2F_b\alpha_1}{3}.
    \end{align}
    Similarly, we integrate the term that accounts for density variation
    \begin{align}
    	& \dfrac{gh^2}{\rho} \dfrac{\partial \rho}{\partial x} \int_{0}^{1} (1-\zeta) \phi_1  \, d\zeta 
    	= \dfrac{gh^2}{\rho} (\rho_s-\rho_w)\dfrac{\partial c_m}{\partial x} \left(- \int_{0}^{1} \zeta \phi_1  \, d\zeta \right)
    	= \dfrac{gh^2}{6\rho} (\rho_s-\rho_w)\partial_x c_m.
    \end{align}
    and the friction term for the additional linear moment equation,
	    \begin{align}
	    \dfrac{1}{\rho}\int_0^1\phi_1\partial_\zeta\sigma_{xz}\,d\zeta
	    &=\dfrac{1}{\rho}\int_0^1\partial_\zeta(\phi_1\sigma_{xz})\,d\zeta
	      -\dfrac{1}{\rho}\int_0^1\sigma_{xz}\partial_\zeta\phi_1\,d\zeta\notag\\
	    &=-\epsilon |u_b|u_b
	      -\dfrac{\nu}{h}\int_0^1\alpha_1(\phi_1^\prime)^2\,d\zeta\notag\\
	    &=-\epsilon |u_b|u_b-\dfrac{4\nu}{h}\alpha_1 .
	    \end{align}
    After putting all together, we obtain the higher average momentum balance in $x$ direction \eqref{final_higher_average_equation}.

\end{document}